\documentclass[11pt]{amsart}
\usepackage{amssymb,amsmath}
\usepackage{caption}

\usepackage{epsfig}
\usepackage{showkeys}

\def\Bbb{\mathbb}

\def\Dt{\partial_t}

\def\eb{\varepsilon}

\def\R {\mathbb{R}}

\def\<{\left<}
\def\>{\right>}

\def\Nx{\nabla_x}
\def\Dx{\Delta_x}
\def\({\left(}
\def\){\right)}
\def\supp{\operatorname{supp}}

\def\divv{\operatorname{div}}
\def\rot{\operatorname{curl}}

\def\R{\Bbb R}
\def\Dx{\Delta_x}
\def\Nx{\nabla_x}
\def\Dt{\partial_t}
\def\divv{\operatorname{div}}

\newtheorem{proposition}{Proposition}[section]
\newtheorem{theorem}[proposition]{Theorem}
\newtheorem{corollary}[proposition]{Corollary}
\newtheorem{lemma}[proposition]{Lemma}

\theoremstyle{definition}
\newtheorem{definition}[proposition]{Definition}

\newtheorem{remark}[proposition]{Remark}

\numberwithin{equation}{section}

\def \no#1#2#3 {{\bf #1} (#3), #2.}
\def \eds#1#2#3 {#1, #2, #3.}
\begin{document}
\title[] {Infinite energy solutions for damped Navier-Stokes equations in $\R^2$}

\author[] {Sergey Zelik}

\thanks{
The author would like to thank Gregory Seregin and Thierry Gallay  for the fruitful
discussions.}
\address{University of Surrey, Department of Mathematics, \newline
Guildford, GU2 7XH, United Kingdom.}
\email{s.zelik@surrey.ac.uk}

\begin{abstract} We study the so-called damped Navier-Stokes equations in the whole 2D space. The global well-posedness, dissipativity and further regularity of weak solutions of this problem in the uniformly-local spaces are verified based on the further development of the weighted energy theory for the Navier-Stokes type problems. Note that any divergent free vector field $u_0\in L^\infty(\R^2)$ is allowed and no assumptions on the spatial decay of solutions as $|x|\to\infty$ are posed. In addition, applying the developed theory  to the case of the classical Navier-Stokes problem in $\R^2$, we show  that the properly defined weak solution can grow at most polynomially (as a quintic polynomial) as time goes to infinity.
\end{abstract}

 \maketitle
\tableofcontents
\section{Introduction}\label{s0}
We study the following damped Navier-Stokes system in the whole space $x\in\R^2$:
\begin{equation}\label{0.1}
\begin{cases}
\Dt u+(u,\Nx)u=\Dx u-\alpha u+\Nx p+g,\\ \divv u=0,\ \
u\big|_{t=0}=u_0,
\end{cases}
\end{equation}
where $\alpha$ is a positive parameter. These equations describe, for instance, a 2-dimensional viscous liquid moving on a rough surface and are
used in geophysical models for large-scale processes in atmosphere and ocean. The term $\alpha u$
parameterizes the extra dissipation occurring in the planetary boundary layer (see, e.g., \cite{Ped79}; see
also \cite{BP06} for the alternative source of damped Euler equations).
\par
The mathematical theory of damped Navier-Stokes and Euler equations is of a big current interest, see \cite{CV08,CVZ11,CR07,Il91,IlT06,IMT04} and references therein. However, most part of these papers study either the case of bounded underlying domain (e.g., with periodic boundary conditions) or the case of finite energy solutions in the whole space $\R^2$ and very few is known about the  infinite-energy solutions (for instance, starting with $u_0\in L^\infty(\R^2)$ or from the proper uniformly local Sobolev space) which  are natural from the physical point of view (for instance, for applications to the so-called uniform turbulence theory or/and statistical hydrodynamics, see \cite{VF80} for the details). Indeed, the main technical tool for the mathematical study of such solutions is the vorticity equation
\begin{equation}\label{0.vort}
\Dt\omega+(u,\Nx)\omega=\Dx\omega-\alpha\omega+\rot g,\ \ \omega=\rot u
\end{equation}
which possesses (in 2D) the maximum principle and, as a consequence, the dissipative $L^\infty$-estimate for the vorticity is almost immediate:
\begin{equation}\label{0.max}
\|\rot u(t)\|_{L^\infty}\le \|\rot u_0\|_{L^\infty}+\frac 1{\alpha}\|\rot g\|_{L^\infty}.
\end{equation}
Nevertheless, estimating the $L^\infty$ or uniformly bounded norms of the velocity field $u$ based on the vorticity estimate \eqref{0.max} is far from being straightforward (in particular, due to the ill-posedness of the Helmholtz projectors in the $L^\infty$-type spaces) and, to the best of our knowledge, only exponentially growing in time estimates of $u(t)$ can be obtained based solely on the vorticity estimate \eqref{0.max}. The situation becomes even worse in the case of classical Navier-Stokes problem ($\alpha=0$), where the vorticity estimate may grow linearly in time, but the best known estimate of the $L^\infty$-norm of $u(t)$ is super-exponential:
\begin{equation}\label{0.bad}
\|u(t)\|_{L^\infty}\le C_1e^{C_2 t^2},
\end{equation}
where $C_1$ and $C_2$ depend on $u_0$ and $g$, see \cite{GMS01,ST07} (see also \cite{AM05} for the polynomial estimates in the case of periodic boundary conditions in one space direction).
\par
Thus, the alternative/complementary to the vorticity estimate ideas/methods
are necessary to obtain more realistic upper bounds for the solution $u(t)$. The most natural would be to use the so-called {\it weighted} energy estimates which are known as very effective tools in the general theory of dissipative PDEs in unbounded domains, see \cite{MZ08,Z03} and the references therein. The main obstacle to the use of such estimates for the Navier-Stokes type problems is the fact that the inertial term $(u,\Nx)u$ does not vanish being multiplied by the weighted term $u\varphi$ (after the integration by parts, here $\varphi(x)$ is a weight function), but produces the non-sign definite cubic term of the form $\varphi' u^3$. Since the other terms in the weighted estimate are at most quadratic, it was not clear how to control this higher order non sign-definite term and to close the estimate. Another problem is that $u\varphi$ is not divergence free, so the pressure term also does not vanish and requires an extra accuracy.
\par
Both of these problems have been overcome in \cite{Z07} (see also \cite{Z08}) for the case of the Navier-Stokes problem in a strip. The key idea there was to use the special weights
$$
\varphi_\eb(x):=(1+\eb^2|x-x_0|)^{-2}
$$
 depending on a small parameter $\eb$ and this small parameter is allowed to depend on the solution. Then, the extra cubic term $\varphi_\eb'u^3\sim \eb\varphi u^3$ can be made small by the proper choice of the parameter $\eb$ depending on the "size" of the solution $u$. The second problem related with the pressure has been overcome by multiplying the equation by $u\varphi-v_\varphi$ where $v_\varphi$ is a small corrector which makes this multiplier divergence free and which can be found as a solution of the proper linear conjugate equation, see \cite{Z07} for the details.
\par
The main aim of the present paper is to extend the weighted energy method to the case of (damped) Navier-Stokes equations in the whole space. Note that the analytic structure of the pressure term is essentially different than in the case of a strip with Dirichlet boundary conditions, so the method of \cite{Z07} does not work in a direct way and requires an essential modification. In particular, we have to handle the pressure term in a completely different way based on the explicit formula for $\nabla p$ through $u\otimes u$ via the  convolution operator and on the trick with the proper splitting of the convolution kernel developed in \cite{L02} which allows us to treat $\nabla p$ in weighted and uniformly local spaces. Thus, combining
the ideas from \cite{Z07} and \cite{L02}, we will give a comprehensive study of the damped Navier-Stokes equations in the uniformly local spaces.
\par
The main result of the paper is the global well-posedness and dissipativity of the damped Navier-Stokes equations \eqref{0.1} in the uniformly local space $L^2_b(\R^2)$ (see Section \ref{s0.5} for the definitions). Namely, if $g\in L^2_b(\R^2)$ is such that $\divv g=0$ and $\rot g\in L^\infty(\R^2)$, then for every divergence free $u_0\in L^2_b(\R^2)$ there exists a unique global weak solution $u(t)\in L^2_b(\R^2)$ of \eqref{0.1}and the following estimate holds:
\begin{equation}\label{0.main}
\|u(t)\|_{L^2_b}\le Q(\|u_0\|_{L^2_b})e^{-\beta t}+Q(\|g\|_{L^2_b}+\|\rot g\|_{L^\infty}),\ \ t\ge0,
\end{equation}
where positive constant $\beta$ and monotone increasing function $Q$ are independent of $t$ and $u$. This dissipative estimate allows us to apply the highly developed general theory of dissipative PDEs for the further study of the damped Navier-Stokes equation \eqref{0.1}. In the present paper, we restrict ourselves only to verifying the existence of the so-called locally compact global attractor for this problem, further applications will be considered in the forthcoming paper.
\par
Although the most part of the paper is devoted to the damped case $\alpha>0$, we also consider the applications of the developed theory to the classical Navier-Stokes problem where $\alpha=0$. In this case, as elementary examples show, the solution $u(t)$ may grow at least linearly as time goes to infinity, so the long standing open problem here is to obtain the realistic/sharp upper bounds for the growth as $t\to\infty$. Although we did not solve this problem completely, the developed method allows us to improve drastically the best known upper bounds \eqref{0.bad} and to show that  the solution $u(t)$ can grow not faster than quintic polynomial in time:
\begin{equation}\label{0.good}
\|u(t)\|_{L^2_b}\le C(t+1)^5,
\end{equation}
where the constant $C$ depends on $u_0$ and $g$, but is independent of $t$. Moreover, we obtain {\it optimal} bounds for the growth rate in time of the mean values of the solution $u(t)$ with respect to the balls of time-dependent radiuses  $R(t):=(t+1)^4$:
\begin{equation}\label{0.good1}
\frac1{R(t)^2}\int_{|x-x_0|\le R(t)}|u(t,x)|^2\,dx\le C(t+1),\ \ x_0\in\R^2
\end{equation}
(this quantity indeed grows linearly in time for simple examples of spatially homogeneous solutions), see Section \ref{s6} for more details.
\par
The paper is organized as follows. In Section \ref{s0.5}, we recall the definitions and basic properties of the weighted and uniformly local Sobolev spaces, introduce special classes of weights and derive a number of elementary inequalities which will be used throughout of the paper.
\par
In Section \ref{s1}, we introduce (inspired mainly by \cite{L02}) a number of technical tools which allows us to treat the pressure term in the proper weighted and uniformly local spaces as well as to exclude it from the various weighted energy estimates.
\par
Section \ref{s3} is devoted to the case of infinite-energy solutions which however {\it decay} to zero as $|x|\to\infty$ (although the rate of decay can be arbitrarily slow). In this case, we are able to verify the global in times boundedness of solutions based solely on the weighted energy estimates (without using the vorticity equation \eqref{0.vort}).
\par
In Section \ref{s4}, we treat the general case of spatially-non-decaying solutions $u(t)\in L^2_b(\R^2)$. The global in time bounds for such solutions are obtained combining the weighted energy estimates obtained in Section \ref{s3} with the vorticity estimate \eqref{0.max}.
\par
In Section \ref{s5}, we verify the dissipative estimate \eqref{0.main} and show that the solution semigroup associated with equation \eqref{0.1} possesses a locally compact global attractor.
\par
In Section \ref{s6}, we consider the classical Navier-Stokes problem ($\alpha=0$) and derive the polynomial upper bounds for the growth in time of the solution $u(t)$. In particular, estimates \eqref{0.good} and \eqref{0.good1} are verified here.
\par
Finally, some approximations of divergence free vector which are necessary in order to prove the existence of solutions in uniformly local spaces are considered in Appendix 1 (see Section \ref{sA}) and the proof of one important for our purposes interpolation inequality is given in Appendix 2 (see Section \ref{sA2}).


\section{Preliminaries I: Weighted and uniformly local spaces}\label{s0.5}

In this section, we briefly discuss the definitions and basic properties of the weighted and uniformly local Sobolev spaces (see \cite{MZ08,Z07,Z03} for more detailed exposition). We start with the class of admissible weight functions and associated weighted spaces.

\begin{definition}\label{Def0.5.weights} A positive function $\phi(x)$, $x\in\R^2$, is a weight function of exponential growth rate $\mu\ge0$ if
\begin{equation}\label{1.weight}
\phi(x+y)\le C e^{\mu|y|}\phi(x),\ \ x,y\in\R^2.
\end{equation}
The associated weighted Lebesgue space $L^p_\phi(\R^2)$, $1\le p<\infty$, is defined as a subspace of functions belonging to $L^2_{loc}(\R^2)$ for which the following norm is finite:
\begin{equation}\label{1.wspace}
\|u\|^p_{L^p_{\phi}}:=\int_{\R^2}\phi(x)|u(x)|^p\,dx<\infty
\end{equation}
and the Sobolev space $W^{l,p}_\phi(\R^2)$ is the subspace of distributions $u\in \mathcal D'(\R^2)$ whose derivatives up to order $l$ inclusively belong to $L^p_\phi(\R^2)$ (this works for positive integer $l$ only, for fractional and negative $l$, the space $W^{l,p}_\phi$ is defined using the interpolation and duality arguments, see \cite{EZ01,Z07} for more details).
\end{definition}
The typical examples of weight functions of exponential growth rate are
\begin{equation}\label{1.exweight}
\phi(x):=e^{-\eb|x-x_0|}\ \ \text{or}\ \ \phi(x):=e^{-\sqrt{1+\eb^2|x-x_0|^2}},\ \ \eb\in\R,\ \ x_0\in\R^2.
\end{equation}
Another class of admissible weights of exponential growth rate are the so-called polynomial weights and, in particular, the weight function
\begin{equation}\label{1.theta}
\theta_{x_0}(x):=\frac1{1+|x-x_0|^3},\ \ x_0\in\R^3,
\end{equation}
which will be essentially used throughout of the paper.
\par
Next, we define the so-called uniformly local Sobolev spaces.

\begin{definition}\label{Def1.ul} The space $L^p_b(\R^2)$ is defined as the subspace of functions of $L^p_{loc}(\R^2)$ for which the following norm is finite:
\begin{equation}\label{1.ul}
\|u\|_{L^p_b}:=\sup_{x_0\in\R^2}\|u\|_{L^p(B^1_{x_0})}<\infty
\end{equation}
(here and below $B^R_{x_0}$ denotes the $R$-ball in $\R^2$ centered at $x_0$)
and the space $\dot L^p_b(\R^2)$ is a closed subspace of $L^2_b(\R^2)$ which consists of functions tending to zero as $|x|\to\infty$ in the following sense:
\begin{equation}\label{1.dotul}
\lim_{|x_0|\to\infty}\|u\|_{L^p(B^1_{x_0})}=0.
\end{equation}
The spaces $W^{l,p}_b(\R^2)$ (resp. $\dot W^{l,p}_b(\R^2)$) are defined as subspaces of distributions $u\in\mathcal D'(\R^2)$ whose derivatives up to order $l$ inclusively belong to $L^p_b(\R^2)$ (resp. $\dot L^p_b(\R^2)$).
\par
In slight abuse of notations, we also define the spaces $W^{l,p}_b([0,T],W^{m,q}(\R^2))$ of functions depending on space $x$ and time $t$ variables using the norm:
\begin{equation}
\|u\|_{W^{l,p}([0,T],W^{m,q})}:=\sup_{t\in[0,T-1]}\sup_{x_0\in\R^2}\|u\|_{W^{l,p}([0,T],W^{m,q}(B^1_{x_0}))}<\infty.
\end{equation}
\end{definition}
The next  proposition gives the useful equivalent norms in the weighted Sobolev spaces
\begin{proposition}\label{Prop1.equiv} Let $\phi$ be the weight function of exponential growth rate and let $1\le p<\infty$, $l\in R$ and $R>0$.
 Then,
\begin{equation}\label{1.equiv}
C_1\int_{x_0\in\R^2}\phi(x_0)\|u\|_{W^{l,p}(B^R_{x_0})}^p\,dx_0\le \|u\|_{W^{l,p}_\phi}^p\le C_2\int_{x_0\in\R^2}\phi(x_0)\|u\|_{W^{l,p}(B^R_{x_0})}^p\,dx_0,
\end{equation}
where the constants $C_i$ depend on $R$, $l$ and $p$ and the constants $C$ and $\mu$ from \eqref{1.weight}, but are independent of $u$ and of the concrete choice of the weight $\phi$.
\end{proposition}
For the proof of these estimates, see e.g., \cite{EZ01}.
\par
Thus, the norms $\int_{x_0\in\R^2}\phi(x_0)\|u\|_{W^{l,p}(B^R_{x_0})}^p\,dx_0$ computed with different $R$'s are equivalent.
\par
The next Proposition gives relations between the weighted and uniformly local norms.

\begin{proposition}\label{Prop1.more} Let $\phi$ be the weight  of exponential growth rate such that $\int_{x\in\R^2}\phi\,dx<\infty$. Then, for every $u\in W^{l,p}_b(\R^2)$ and every $\kappa\ge 1$,
\begin{equation}\label{1.ul-w}
\|u\|_{W^{l,p}(B^\kappa_{x_0})}^p\le C\int_{y\in B^\kappa_{x_0}}\|u\|_{W^{l,p}(B^1_y)}^p\,dy\le C_\kappa\int_{y\in\R^2}\phi(y-x_0)\|u\|^p_{W^{l,p}(B^1_{y})}\,dy
\end{equation}
and, in particular, fixing $\kappa=1$ in \eqref{1.ul-w} and taking the supremum with respect to $x_0\in\R^2$, we have
\begin{equation}\label{1.ul-w1}
\|u\|_{W^{l,p}_b}\le C\sup_{x_0\in\R^2}\|u\|_{W^{l,p}_{\phi(\cdot-x_0)}},
\end{equation}
where $C$ is independent of $u$ and the concrete choice of the weight $\phi$. In addition,
\begin{equation}\label{1.w-ul}
\|u\|_{W^{l,p}_\phi}^p\le C\|\phi\|_{L^1}\|u\|_{W^{l,p}_b}^p,
\end{equation}
where $C$ is also independent of $u$ and the concrete choice of $\phi$.
\end{proposition}
For the proof of these results, see e.g., \cite{Z07,Z03}.
\par
The next lemma gives a simple, but important estimate for the weights $\theta_{x_0}(x)$ which will allow us to handle the convolution operators in weighted spaces.
\begin{lemma}\label{Lem1.weight} Let $\theta_{x_0}(x)$ be the weight defined via \eqref{1.theta}. Then, the following estimate holds:
\begin{equation}\label{1.thetakey}
\int_{x\in\R^2}\theta_{x_0}(x)\theta_{y_0}(x)\,dx\le C\theta_{x_0}(y_0),
\end{equation}
where $C$ is independent of $x_0,y_0\in\R^2$.
\end{lemma}
\begin{proof}Indeed, due to the triangle inequality,
$$
2+|x-x_0|^3+|x-y_0|^3\ge\frac14(1+(|x-x_0|+|x-y_0|)^3)\ge \frac14(1+|x_0-y_0|^3).
$$
Therefore,
\begin{multline*}
\theta_{x_0}(x)\theta_{y_0}(x)=\frac1{2+|x-x_0|^3+|x-y_0|^3}\(\frac1{1+|x-x_0|^3}+\frac1{1+|x-y_0|^3}\)\le\\\le
4\theta_{x_0}(y_0)(\theta_{x_0}(x)+\theta_{y_0}(x))
\end{multline*}
and the integration of this inequality with respect to $x$ gives the desired estimate \eqref{1.thetakey}. Thus, the lemma is proved.
\end{proof}

\begin{corollary}\label{Cor1.conv} Let $\theta_{x_0}(x)$ be defined via \eqref{1.theta}. Then, for every $u\in L^p_{\theta_{x_0}}(\R^2)$, we have
\begin{equation}\label{1.inv}
\|u\|_{L^p_{\theta_{y_0}}}^p\le C\int_{x_0\in\R^2}\phi_{y_0}(x_0)\|u\|^p_{L^p_{\theta_{x_0}}}\,dx_0,
\end{equation}
where $C$ is independent of $y_0\in\R$. Moreover, if the function $v\in L^p_{loc}(\R^2)$ satisfies the estimate
\begin{equation}\label{1.convol}
|v(y)|\le C\int_{x\in\R^2}\theta_{y}(x)|u(x)|\,dx,\ \ \forall y\in\R^2,
\end{equation}
then $v\in L^p_{\theta_{y_0}}(\R^2)$ and
\begin{equation}\label{1.convest}
\|v\|_{L^p_{\theta_{y_0}}}\le C_1\|u\|_{L^p_{\theta_{y_0}}},
\end{equation}
where $C_1$ depends on $C$, but is independent of $y_0\in\R^2$.
\end{corollary}
\begin{proof} Indeed, \eqref{1.inv} is an immediate corollary of \eqref{1.thetakey} and the Fubini theorem. To verify \eqref{1.convest}, we note that, due to the H\"older inequality, estimate \eqref{1.convol} implies that
$$
|v(y)|^p\le \tilde C\int_{x\in\R^2}\theta_{y}(x)|u(x)|^p\,dx.
$$
Multiplying this inequality by $\theta_{y_0}(y)$, integrating over $y$ and using the Fubini theorem and \eqref{1.thetakey}, we have
\begin{multline*}
\|v\|_{L^p_{y_0}}^p\le \tilde C\int_{x\in\R^2}\int_{y\in\R^2}\theta_{y_0}(y)\theta_{x}(y)\,dy |u(x)|^p\,dx\le C_1\int_{x\in\R^2}\theta_{y_0}(x)|u(x)|^p\,dx=C_1\|u\|^p_{L^p_{\theta_{y_0}}}
\end{multline*}
and the corollary is proved.
\end{proof}
We conclude this section by introducing some weights and norms depending on a big parameter $R$ which will be crucial for what follows. First, we introduce the following equivalent norm in the space $W^{l,p}_b(\R^2)$:
\begin{equation}
\|u\|_{W^{l,p}_{b,R}}:=\sup_{x_0\in\R^2}\|u\|_{W^{l,p}(B^R_{x_0})}.
\end{equation}
Then, according to \eqref{1.ul-w},
\begin{equation}\label{1.ulR}
\|u\|_{W^{l,p}_b}\le\|u\|_{W^{l,p}_{b,R}}\le CR^{2/p}\|u\|_{W^{l,p}_b},
\end{equation}
where the constant $C$ is independent of $R\ge1$. We also introduce the scaled weight function
\begin{equation}\label{ttheta}
\theta_{R,x_0}(x):=\frac1{R^3+|x-x_0|^3}=R^{-3}\theta_{x_0/R}(x/R).
\end{equation}
Then, the scaled analogue of \eqref{1.ul-w} reads
\begin{equation}\label{1.ul-wR}
\|u\|_{W^{l,p}(B^{\kappa R}_{x_0})}^p\le CR^{-2}\int_{y\in B^{\kappa R}_{x_0}}\|u\|_{W^{l,p}(B^R_y)}^p\,dy\le C_\kappa R\int_{y\in\R^2}\theta_{R,x_0}(y)\|u\|^p_{W^{l,p}(B^R_{y})}\,dy,
\end{equation}
where the constants $C$ and $C_\kappa$ are independent of $R$ and the scaled analogue of \eqref{1.thetakey}
\begin{equation}\label{1.R-thetakey}
\int_{x\in\R^2}\theta_{R,x_0}(x)\theta_{R,y_0}(x)\,dx\le CR^{-1}\theta_{R,x_0}(y_0),
\end{equation}
where $C$ is independent of $x_0,y_0\in\R^2$ and $R>0$. Moreover,  multiplying  inequality \eqref{1.ul-wR} by $\theta_{R,y_0}(x_0)$, integrating over $x_0$ and using \eqref{1.R-thetakey}, we see that
\begin{equation}\label{1.eqtheta}
\int_{x\in\R^2}\theta_{R,x_0}(x)\|u\|^p_{W^{l,p}(B^{\kappa R}_x)}\,dx\le C_\kappa \int_{x\in\R^2}\theta_{R,x_0}(x)\|u\|^p_{W^{l,p}(B^{R}_x)}\,dx,
\end{equation}
where $C_\kappa$ is independent of $R$. We also note that, analogously to \eqref{1.w-ul} and using \eqref{1.ulR},
\begin{equation}\label{ulR}
\int_{x\in\R^2}\theta_{R,x_0}(x)\|u\|^p_{W^{l,p}(B^R_{x})}\,dx\le CR^{-1}\|u\|^p_{W^{l,p}_{b,R}}\le C_1R\|u\|^p_{W^{l,p}_b},
\end{equation}
where the constants $C$ and $C_1$ are independent of $R\gg1$. Finally, we need one more simple statement about the behavior of the quantities involved into \eqref{ulR} in the case when $u\in \dot W^{l,p}_b(\R^2)$.

\begin{proposition}\label{Prop1.decay} Let $u\in\dot W^{l,p}_b(\R^2)$. Then,
\begin{equation}\label{limzero}
\lim_{R\to\infty}\frac1{R^2}\|u\|_{W^{l,p}_{b,R}(\R^2)}=0.
\end{equation}
\end{proposition}\nopagebreak
The proof of this proposition is straightforward, so we leave it to the reader.

\section{Preliminaries II: Excluding the pressure} \label{s1}
In this section, we introduce the key estimates which allow us to work with the pressure term $\nabla p$ in the uniformly local spaces. Note that the Helmholtz decomposition does not work for the general vector fields belonging to $L^2_b(\R^2)$, so the standard (for the bounded domains) approach
does not work at least directly and we need to proceed in a bit more accurate way.
\par
As usual, we assume that \eqref{0.1} is satisfied in the sense of distributions. Then, taking  the divergence from both sides of \eqref{0.1} and assuming that the external forces $g$ are divergence free:
\begin{equation}\label{gdiv}
\divv g=0,
\end{equation}
we have
\begin{equation}\label{0.5}
\Dx p=\divv((u,\Nx)u)=\sum_{i,j=1}^2\partial_{x_i}\partial_{x_j}(u_iu_j).
\end{equation}
Thus, formally, $p$ can be expressed through $u$ by the following singular integral operator:
\begin{equation}\label{0.6}
p(y):=\int_{\R^2}\sum_{ij}K_{ij}(x-y)u_i(x)u_j(x)\,dx,\ \ K_{ij}(x):=\frac1{2\pi}\frac{|x|^2\delta_{ij}-2x_ix_j}{|x|^4}
\end{equation}
which we present in the form
\begin{equation}\label{0.conv}
p=\Bbb K w:=K*w,\ \ w:=u\otimes u,\ \ K*w=\sum_{ij} K_{ij} * w_{ij}.
\end{equation}
It is well-known that the convolution operator $\Bbb K$ is well-defined as a bounded  linear operator from $w\in [L^{q}(\R^2)]^4$ to $p\in L^q(\R^2)$,
$1<q<\infty$, but it is not true neither for $q=\infty$ nor for the uniformly-local space $L^q_b(\R^2)$. However, as the following simple lemma shows, the {\it gradient} of $p$ (which is sufficient in order to define a solution of \eqref{0.1}) is
well-defined in uniformly-local spaces and has natural regularity properties.

\begin{lemma}\label{Lem0.1} The operator $w\to\Nx p$, where $p$ is defined via \eqref{0.conv} can be extended by continuity (in $L^p_{loc}$) in a unique way to the bounded operator from $[L^q_b(\R^2)]^4$ to $[W^{-1,q}_b(\R^2)]^2$, $1<q<\infty$.
\end{lemma}
\begin{proof} Indeed, let $\psi(x)\in C^\infty_0(\R^2)$ be the cut-off function being identically one near $x=0$ (for $|x|\le1/4$)
  and identically zero for $|x|\ge1/2$ and let $w\in [C_0^\infty(\R^2)]^4$. Then,
\begin{multline}\label{0.7}
\nabla_y p=\nabla_y \sum_{ij}\int_{\R^2}\psi(x-y)K_{ij}(x-y)w_{ij}(x)\,dx+\\+\sum_{ij}\int_{\R^2}\nabla_y[(1-\psi(x-y))K_{ij}(x-y)]w_{ij}(x)\,dx:=
\Bbb K_1 w+\Bbb K_2 w.
\end{multline}
The first operator is extendable to the uniformly local spaces since the kernel vanishes if $|x|\ge1/2$ and by the standard Zygmund-Calderon estimates (see, e.g., \cite{Tri78}), we have
\begin{equation}\label{0.8}
\|\Bbb K_1 w\|_{W^{-1,q}(B_{x_0}^1)}^q\le C\|w\|_{L^q(B^{3/2}_{x_0})}^q,
\end{equation}
where $C$ is independent of $x_0\in\R^2$. Thus, multiplying \eqref{0.8} by $\phi(x_0)$, integrating over $x_0\in\R^3$ and using \eqref{1.ul-w} and \eqref{1.equiv}, we see that
\begin{equation}\label{0.1good}
\|\Bbb K_1 w\|_{W^{-1,q}_\phi}\le C_\phi\|w\|_{L^q_\phi},
\end{equation}
holds for every weight function $\phi$ of exponential growth rate. Moreover, taking $\phi=\phi_{x_0}(x):=\phi(x-x_0)$ in \eqref{0.1good} where $\phi$ is an integrable weight of exponential growth rate, taking the supremum over $x_0\in\R^2$ and using \eqref{1.ul-w1} and \eqref{1.w-ul}, we see that the operator $\Bbb K_1$ can be extended in a unique way on functions $w\in [L^q_b(\R^2)]^4$ with the natural estimate
\begin{equation}\label{0.1ggood}
\|\Bbb K_1w\|_{W^{-1,q}_b}\le C\|w\|_{L^q_b}.
\end{equation}

 The second operator of \eqref{0.7} is a smoothing  convolution operator with integrable kernel (decaying as $\frac1{|x|^3}$  as $|x|\to\infty$), namely,
$$
|K_2(z)|\le\frac C{1+|z|^3}
$$
for some constant $C$ independent of $z$. Therefore, according to Corollary \ref{Cor1.conv},
\begin{equation}\label{0.2ggood}
\|\Bbb K_2 w\|_{L^q_{\theta_{x_0}}}\le C\|w\|_{L^q_{\theta_{x_0}}},\ \ \ \theta_{x_0}(x):=\frac1{1+|x-x_0|^3},
\end{equation}
where $C$ is independent of $x_0\in\R^2$. Thus, taking the supremum over $x_0\in\R^2$, we see that $\Bbb K_2$ can also be extended by continuity to $w\in[L^q_b(\R^2)]^2$ and
$$
\|\Bbb K_2w\|_{L^q_b}\le C\|w\|_{L^q_b}
$$
and the lemma is proved.
\end{proof}
We will denote the operator obtained in the lemma by $\Nx P$ and the corresponding term in the Navier-Stokes equation will be denoted by $\Nx P(u\otimes u)$. Then, in particular
\begin{equation}\label{0.9}
\Nx P(u\otimes u): [L^4_b(\R^2)]^2\to [W^{-1,2}_b(\R^2)]^2.
\end{equation}
We are now ready to give the definition of a weak solution of problem \eqref{0.1}.
\begin{definition}\label{Def0.1} Let the external forces $g\in L^2_b(\R^2)$ be divergent free. A  function $u(t,x)$ is a weak solution of problem \eqref{0.1} if
\begin{equation}\label{0.10}
u\in L^\infty([0,T],L^2_b(\R^2)),\ \ \ \Nx u\in L^2_b([0,T]\times\R^2)
\end{equation}
and the equation is satisfied in the sense of distributions with $\Nx p=\Nx P(u\otimes u)$ defined in Lemma \ref{Lem0.1}. Note that, according to the embedding theorem, $u\in L^4_b([0,T]\times\R^2)$ and $u\otimes u\in L^2_b([0,T]\times\R^2)$ and, due to the previous lemma, $\Nx p\in L^2_b([0,T],W^{-1,2}_b(\R^2))$ and the equation \eqref{0.1} can be understood as equality in this space.
\end{definition}
\begin{remark} We emphasize once more that only the gradient of pressure $\nabla p$ is well-defined as an element of $L^2_b([0,T],W^{-1,2}_b)$, but the pressure itself may be unbounded as $|x|\to\infty$. To be more precise, the operator $\nabla P$ defined above satisfies
\begin{equation}\label{1.press}
\divv \nabla P(w)=\sum_{ij}\partial_{x_i}\partial_{x_j} w_{ij},\ \ \rot \nabla P(w)=0
\end{equation}
in the sense of distributions. These relations can be justified by approximating $w$ by finite
functions and passing to the limit analogously to Lemma \ref{Lem0.1}. Therefore, there is a function $p\in L^2([0,T],L^2_{loc}(\R^2))$ such that
$$
\nabla p=\nabla P(w),
$$
see \cite{Tem77}, but this function may grow as $|x|\to\infty$. Note also that the choice of $\nabla p=\nabla P(u\otimes u)$ is not unique. However, if $p_1$ and $p_2$ both satisfy \eqref{0.5} (for the same velocity field $u$), then the difference $p_1-p_2$ solves $\Delta(p_1-p_2)=0$ (in the sense of distributions) and, consequently is a {\it harmonic} function. Moreover, every harmonic function with bounded gradient is linear, so, if we want the velocity field $u$ to be in the proper uniformly local space, the most general choice of the pressure is
\begin{equation}
\nabla p=\nabla P(u\otimes u)+\vec C(t),
\end{equation}
where the constant vector $\vec C(t)$ depends only on time (and is independent of $x$) and $\nabla P$ is defined in Lemma \ref{Lem0.1}. In the present paper, we consider only the choice $\vec C(t)\equiv0$. In a fact, the vector $\vec C(t)$ should be treated as one more external data and can be chosen arbitrarily, but this does not lead to more general theory since everything can be reduced to the case of $\vec C\equiv 0$ by replacing the external force $g$ by $g-\vec C(t)$.
\end{remark}
\begin{remark}\label{Rem.non-div} Assumption \eqref{gdiv} on the external forces $g$ can be slightly relaxed. Namely, we may assume instead that
\begin{equation}\label{g.proj}
\Pi g\in [L^2_b(\R^2)]^2,
\end{equation}
where $\Pi$ is the proper extension of the Leray projector to the divergent free vector fields. In this case, instead of \eqref{0.5}, we have
$$
\Dx p=\sum_{j,j}\partial_{x_i}\partial_{x_j}(u_iu_j)-\divv g
$$
which gives
$$
\Nx p=\Nx P(u\otimes u)+\Nx(\Dx)^{-1}\divv g=\Nx P(u\otimes u)-(1-\Pi)g
$$
and, inserting this expression to \eqref{0.1}, we see that the problem is factually reduced to the case of divergent free external forces by replacing $g$ by $\Pi g$. However, even the assumption $g\in [L^\infty(\R^2)]^2$ (or $[C^\infty_b(\R^2)]^2$) does not imply that $\Pi g\in [L^2_b(\R^2)]^2$. In this case, we can only guarantee that $\Pi g\in [BMO(\R^2)]^2$, but the functions with bounded mean oscillation may grow as $|x|\to\infty$ (say, as  $\log|x|$, see \cite{L02}), so \eqref{g.proj} is not satisfied automatically.
\par
Note also that this condition is {\it necessary} if we want the solution $u(t)$ to be bounded as $|x|\to\infty$ (and belong to the proper uniformly local Soboloev space). Indeed, after replacing the pressure term, \eqref{0.1} reads
$$
\Dt u+(u,\Nx)u=\Dx u-\alpha u+\Nx P(u\otimes u)+\Pi g
$$
and this equation cannot have bounded (as $|x|\to\infty$) solutions $u(t)$ if $\Pi g$ is not bounded (as $|x|\to\infty$).
\end{remark}
We conclude this preliminary section by establishing the key estimate which allows us handle the pressure term in weighted energy estimates.
To this end, we introduce for every $x_0\in\R^2$ and $R>1$ the
cut-off function $\varphi_{R,x_0}$ which satisfies
\begin{equation}\label{0.11}
\varphi_{R,x_0}(x)\equiv1,\ for \ x\in B^R_{x_0},\ \ \ \varphi_{R,x_0}(x)\equiv 0,\ for \ x\notin B^{2R}_{x_0},
\end{equation}
and
\begin{equation}\label{0.12}
|\Nx\varphi_{R,x_0}(x)|\le CR^{-1}\varphi^{1/2}_{R,x_0}(x),
\end{equation}
where $C$ is independent of $R$ (obviously such family of cut-off functions exist). Then, the following result holds.

\begin{lemma}\label{Lem0.2} Let the exponents $1<p,q<\infty$, $\frac1p+\frac1q=1$, $w\in [L^p_b(\R^2)]^4$ and $v\in [W^{1,q}(\R^2)]^2$ be
divergence free. Then the following estimate holds:
\begin{equation}\label{0.15}
|(\Nx P(w),\varphi_{R,x_0}v)|\le C\int_{\R^2}\theta_{R,x_0}(x)\|w\|_{L^p(B^R_{x})}\,dx\cdot\|\varphi_{R,x_0}^{1/2}v\|_{L^q},
\end{equation}
where $C$ is independent of $R$ and $x_0$ and $\theta_{R,x_0}(x)$ is defined by \eqref{ttheta}.
\end{lemma}
\begin{proof} Indeed, it is sufficient to verify the formula for $w,v\in C_0^\infty$ only. For general $w$'s and $v$'s belonging to the uniformly local spaces, the estimate can be justified in a standard way using the approximations. For the functions with finite support, the function $p$ is well-defined, so we may integrate by parts and write
\begin{equation}\label{1.intpart}
(\Nx P(w),\varphi_{R,x_0}v)=-(p(w)-p_{x_0},\Nx\varphi_{R,x_0}\cdot v),
\end{equation}
where $p_{x_0}=p_{x_0}(u,R)$ is a constant which will be specified below. Analogously to \eqref{0.7}, we write
\begin{multline}\label{0.16}
p(y)=  \sum_{ij}\int_{\R^2}\psi_R(x-y)K_{ij}(x-y)w_{ij}(x)\,dx+\\+\sum_{ij}\int_{\R^2}[(1-\psi_R(x-y))K_{ij}(x-y)]w_{ij}(x)\,dx:=
\widetilde K_{1,R}(w)+\widetilde K_{2,R}(w).
\end{multline}
where $\psi_R(x)=\psi(x/R)$ and $\psi$ is the same as in \eqref{0.7}. Actually, \eqref{0.16} coincides with \eqref{0.7} up to scaling $x$ and absence of the gradient. Thus,
using the scaling $x\to Rx$, $x_0\to Rx_0$ and $y\to Ry$, we obtain the operators independent of $R$. Estimating the first one
analogously to \eqref{0.8} and returning back to the unscaled variables, we have
\begin{multline}\label{0.17}
\|\widetilde K_{1,R}(w)\|_{L^p(B^{2R}_{x_0})}\le C\|w\|_{L^p(B^{3R}_{x_0})}\le\\\le C_1 R^{-2}\int_{|x-x_0|<4R}\|w\|_{L^p(B^R_{x})}\,dx\le
  C_2R\int_{\R^2}\theta_{R,x_0}(x)\|w\|_{L^p(B^R_{x})}\,dx,
\end{multline}
where we have also implicitly used \eqref{1.ul-wR}.
\par
The second operator $\widetilde K_{2,R}$ is a bit more delicate. Indeed, in contrast to \eqref{0.7}, the convolution kernel is not differentiated by $y$ and decays only as $\frac1{|x|^2}$ as $|x|\to\infty$ which is insufficient for the uniformly local and weighted estimates. However, following \cite{L02}, we use the trick with the properly chosen constant $p_{x_0}$ which allows to resolve the problem. Namely, we fix
\begin{equation}\label{0.18}
p_{x_0}:=\tilde K_{2,R}(w)(x_0)=\sum_{ij}\int_{\R^2}(1-\psi_R(x-x_0))K_{ij}(x-x_0)w_{ij}(x)\,dx.
\end{equation}
Then, the new convolution operator
$$
\widetilde K_{2,R}(w)-p_{x_0}=\sum\int\bar K_{ij}(x,y)w_{ij}\,dx
$$
has the  kernel
$$
\bar K_{ij}(x,y):=(1-\psi_R(x-y))K_{ij}(x-y)-(1-\psi_R(x-x_0))K_{ij}(x-x_0)
$$
 which decays as $|x-y|^{-3}$ when $|x-y|\to\infty$ and $|y-x_0|$ remains bounded, so the singularity at infinity disappear and the uniformly local estimates work. Indeed, in scaled variables $x=Rx'$, $y=Ry'$, $x_0=Rx_0'$, we obviously have
\begin{equation}\label{0.19}
|\bar K_{ij}(x',y')|\le C\frac1{1+|x'-y'|^3},\ \ |y'-x_0'|\le2
\end{equation}
and, therefore, again in scaled variables
$$
\|\widetilde K_{2,R}(w)-p_{x_0}\|_{L^p(B^{2}_{x_0'})}\le C\int_{x'\in\R^2}\frac1{1+|x'-x_0'|^3}\|w\|_{L^p(B^1_{x'})}\,dx'
$$
and returning back to the non-scaled variables, we see that
\begin{equation}\label{0.20}
\|\widetilde K_{2,R}(w)-p_{x_0}\|_{L^p(B^{2R}_{x_0})}\le  C_2R\int_{\R^2}\frac{1}{R^3+|x-x_0|^3}\|w\|_{L^p(B^R_{x})}\,dx.
\end{equation}
Combining \eqref{0.17} and \eqref{0.20}, we see that
$$
\|p(w)-p_{x_0}\|_{L^p(B^{2R}_{x_0})}\le C R\int_{\R^2}\theta_{R,x_0}(x)\|w\|_{L^p(B^{R}_{x})}\,dx,
$$
where $C$ is independent of $R$ and $x_0$. Finally, applying the H\"older inequality to \eqref{1.intpart}
 and using \eqref{0.12}, we arrive at the desired estimate \eqref{0.15} and finish the proof of the lemma.
\end{proof}

\section{Global well-posedness: the case of  spatially decaying initial data}\label{s3}

In this section, we will study the damped Navier-Stokes equations under the additional assumptions that
\begin{equation}\label{3.gu}
g\in \dot L^2_b(\R^2),\ \  u_0\in \dot L^2_b(\R^2).
\end{equation}
The main result of this section is the following theorem.
\begin{theorem}\label{Th0.1} Let the external forces $g$ and the initial data $u_0$ be divergence free and satisfy \eqref{3.gu}. Then, the damped Navier-Stokes problem
\eqref{0.1} possesses at least one global weak solution in the sense of Definition \ref{Def0.1} and this solution satisfies the estimate
\begin{equation}\label{3.global}
\|u(T)\|_{L^2_b(\R^2)}+\|\Nx u\|_{L^2_b([0,T]\times\R^2)}\le Q(u_0,g),
\end{equation}
where the constant $Q(u_0,g)$ is independent of $T\ge0$.
\end{theorem}

\begin{proof} We give below only the derivation of the key a priori estimate \eqref{3.global}. The existence of a solution can be
obtained after that  in a standard way by approximating the infinite-energy solution by the finite-energy ones and passing to the limit, see \cite{MZ08,Z07} for the details (see also some more details at the end of the proof).
\par
We multiply equation \eqref{0.1} on $\varphi_{R,x_0} u$ where $R$ is a big positive number which will be fixed below, $x_0\in\R^2$,  and the cut-off functions $\varphi_{R,x_0}$ satisfy \eqref{0.11} and \eqref{0.12}.
Then, after the integration over $x$ and standard transformations, we get
\begin{multline}\label{0.13}
\frac12 \frac d{dt}\|u(t)\|_{L^2_{\varphi_{R,x_0}}}^2+\alpha\|u\|^2_{L^2_{\varphi_{R,x_0}}}+\|\Nx(\varphi_{R,x_0}^{1/2}u)\|^2_{L^2}\le\\\le C\|g\|_{L^2_{\varphi_{R,x_0}}}^2+CR^{-2}\|u\|_{L^2(B^{2R}_{x_0})}^2+
 CR^{-1}\|u\|^3_{L^3(B^{2R}_{x_0})}+|(\Nx P(u),\varphi_{R,x_0}u)|.
\end{multline}
Here the constant $C$ is independent of $R$. To estimate the term containing pressure, we use Lemma \ref{Lem0.2} with $q=3$ and $p=3/2$. Then, due to
\eqref{0.15} together with the H\"older and Young inequalities,
\begin{multline}\label{0.huuge}
|(\Nx P(u),\varphi_{R,x_0}u)|\le C\int_{\R^2}\theta_{R,x_0}(x)\|u\otimes u\|_{L^{3/2}(B^R_{x})}\,dx\cdot\|u\|_{L^3(B^{2R}_{x_0})}\le\\\le
C\int_{\R^2}\theta_{R,x_0}(x)\|u\|^2_{L^3(B^{R}_{x})}\,dx\cdot\|u\|_{L^3(B^{2R}_{x_0})}\le \\\le C\(\int_{\R^2}\theta_{R,x_0}(x)\,dx\)^{1/3}\(\int_{\R^2}\theta_{R,x_0}(x)\|u\|^3_{L^3(B^{R}_{x})}\,dx\)^{2/3}\cdot\|u\|_{L^3(B^{2R}_{x_0})}\le\\\le CR^{-1/3}\(\int_{\R^2}\theta_{R,x_0}(x)\|u\|^3_{L^3(B^{R}_{x})}\,dx\)^{2/3}\cdot\|u\|_{L^3(B^{2R}_{x_0})}\le\\\le C\int_{\R^2}\theta_{R,x_0}(x)\|u\|^3_{L^3(B^{R}_{x})}\,dx+
CR^{-1}\|u\|_{L^3(B^{2R}_{x_0})}^3,
\end{multline}
where all constants are independent of $R\gg1$. Thus, \eqref{0.13} now reads
\begin{multline}\label{2.13}
\frac d{dt}\|u(t)\|_{L^2_{\varphi_{R,x_0}}}^2+\alpha\|u\|^2_{L^2_{\varphi_{R,x_0}}}+\|\Nx u\|^2_{L^2(B^{R}_{x_0})}\le\\\le C\|g\|_{L^2_{\varphi_{R,x_0}}}^2+CR^{-2}\|u\|_{L^2(B^{2R}_{x_0})}^2+
 CR^{-1}\|u\|^3_{L^3(B^{2R}_{x_0})}+C\int_{\R^2}\theta_{R,x_0}(x)\|u\|^3_{L^3(B^{R}_{x})}\,dx.
\end{multline}
We  introduce
\begin{equation}\label{0.25}
Z_{R,y_0}(u):=\int_{x_0\in\R^2}\theta_{R,y_0}(x_0)\|u\|^2_{L^2_{\varphi_{R,x_0}}}\,dx_0.
\end{equation}
Then, using \eqref{1.eqtheta}, we have
\begin{equation}\label{3.Z}
C_2\int_{y\in\R^2}\theta_{R,y_0}(y)\|u\|^2_{L^2(B^R_y)}\,dy\le Z_{R,y_0}(u)\le C_1\int_{y\in\R^2}\theta_{R,y_0}(y)\|u\|^2_{L^2(B^R_y)}\,dy,
\end{equation}
where $C_i$ are independent of $R$. Multiplying now equation \eqref{2.13} on $\theta_{R,y_0}(x_0)$, integrating over $x_0\in\R^2$ and using \eqref{1.R-thetakey}, we see that, for sufficiently large $R$,
\begin{multline}\label{3.main}
\frac d{dt}Z_{R,x_0}(u(t))+2\beta Z_{R,x_0}(u(t))+2\beta \int_{x\in\R^2}\theta_{R,x_0}(x)\|u\|_{W^{1,2}(B^R_x)}^2\,dx\le\\\le C Z_{R,x_0}(g)+
 CR^{-1} \int_{x\in\R^2}\theta_{R,x_0}(x)\|u\|_{L^3(B^R_x)}^3\,dx,
\end{multline}
where the positive constants $C$ and $\beta$ are independent of $R$.
\par
Thus, we only  need to estimate the integral containing the $L^3$-norm of $u$ in the right-hand side. To this end, we use the interpolation inequality
\begin{equation}\label{3.int}
\|u\|^3_{L^3(B^R_{x_0})}\le C\|u\|^2_{L^2(B^R_{x_0})}\|u\|_{W^{1,2}(B^R_{x_0})},
\end{equation}
where $C$ is independent of $x_0$ and $R$. Then, using also the H\"older inequality, we end up with
\begin{multline}\label{3.qubic}
CR^{-1} \int_{x\in\R^2}\theta_{R,x_0}(x)\|u\|_{L^3(B^R_x)}^3\,dx\le C R^{-1}\int_{x\in\R^2}\theta_{R,x_0}(x)\|u\|_{L^2(B^R_x)}^2\|u\|_{W^{1,2}(B^R_{x})}\,dx\le\\\le CR^{-1}\|u\|_{L^2_{b,R}} \int_{x\in\R^2}\theta_{R,x_0}(x)\|u\|_{L^2(B^R_x)}\|u\|_{W^{1,2}(B^R_{x})}\,dx\le\\\le CR^{-2}\|u\|^2_{L^2_{b,R}} Z_{R,x_0}(u)+\beta \int_{x\in\R^2}\theta_{R,x_0}(x)\|u\|_{W^{1,2}(B^R_x)}^2\,dx.
\end{multline}
Inserting this estimate into the right-hand side of \eqref{3.main}, we  arrive at
\begin{multline}\label{0.27}
\frac d{dt}Z_{R,x_0}(u(t))+\beta Z_{R,x_0}(u(t))+\beta Z_{R,x_0}(u(t))(1-CR^{-2}\|u(t)\|^2_{L^2_{b,R}})+\\+\beta \int_{x\in\R^2}\theta_{R,x_0}(x)\|u\|_{W^{1,2}(B^R_x)}^2\,dx\le C Z_{R,x_0}(g)),
\end{multline}
where the positive constants $C$ and $\beta$ are independent of $R$.
\par
Finally, due to \eqref{1.ul-wR}
\begin{equation}\label{0.281}
\|u\|^2_{L^2_{b,R}}\le CR\sup_{x_0\in\Bbb R^2}Z_{R,x_0}(u),
\end{equation}
and \eqref{0.27} reads
\begin{multline}\label{3.last}
\frac d{dt}Z_{R,x_0}(u(t))+\beta Z_{R,x_0}(u(t))+\beta Z_{R,x_0}(u(t))(1-KR^{-1}\sup_{x_0\in\R}Z_{R,x_0}(u(t)))+\\+\beta \int_{x\in\R^2}\theta_{R,x_0}(x)\|u\|_{W^{1,2}(B^R_x)}^2\,dx\le C Z_{R,x_0}(g)),
\end{multline}
where the positive constants $K$, $\beta$ and $C$ are independent of $R\gg1$.
\par
We claim that the key estimate \eqref{3.last} is sufficient to finish the proof of the theorem.
  Indeed, \eqref{3.last} gives the desired estimate
\begin{multline}\label{0.28}
Z_{R,x_0}(u(t))+\int_t^{t+1}\|u(t)\|^2_{L^2_{\theta_{R,x_0}}}\,dt\le\\\le C Z_{R,x_0}(u_0)e^{-\beta t}+C Z_{R,x_0}(g)\le C(Z_{R,x_0}(u_0)+Z_{R,x_0}(g))
\end{multline}
(which implies \eqref{3.global} just by taking the supremum via $x_0\in\R^2$)
if we are able to fix $R=R(u_0,g)$ in such way that
\begin{equation}\label{0.29}
R^{-1}\sup_{x_0\in\R}Z_{R,x_0}(u(t))\le\frac1K\ \
\end{equation}
holds for every $t\in\R_+$.
\par
On the other hand, due to the decay assumption $u_0,g\in \dot L^2_b(\R^2)$ and Proposition
\ref{Prop1.decay}, we may fix $R=R(u_0,g)$ in such way that
\begin{equation}\label{0.30}
CR^{-1}(1+Z_{R,x_0}(u_0)+Z_{R,x_0}(g))\le\frac1{2K},\ \ x_0\in\Bbb R^2,
\end{equation}
where $C$ and $K$ are the same as in \eqref{0.28} and \eqref{3.last} respectively.
\par
Then \eqref{0.29} is satisfied for $t=0$ and the standard continuity arguments show that both \eqref{0.29} and  \eqref{0.28} hold for all $t\in\R_+$
if we a priori know that the function
\begin{equation}\label{3.cont}
t\to\sup_{x_0\in\R^2} Z_{R,x_0}(u(t))
\end{equation}
is {\it continuous} as a function of time. Thus, the desired estimate \eqref{3.global} is verified under the additional continuity assumption for the function \eqref{3.cont}.
\par
However, in contrast to the case of finite energy solutions, the regularity of a weak solution stated in Definition \ref{Def0.1} is not enough to establish that $u\in C([0,T], L^2_b(\R^2))$, so a priori, function \eqref{3.cont} may be not continuous (in a fact, it is continuous a posteriori due to the decay assumption $u_0,g\in\dot L^2_b(\R^2)$, but in order to verify that, we need some extra estimates).
\par
We overcome the continuity problem in an alternative way (following \cite{Z07}), namely, we cut-off the external force $g$ and consider the finite support external forces $g^N(x):=g(x)\varphi_{N,0}(x)$ (where $\varphi_{R,x_0}(x)$ are defined by \eqref{0.11} and \eqref{0.12}), approximate the divergence free initial data $u_0$ by the sequence $u_0^N$ of the divergence free vector fields with finite supports using Corollary \ref{CorA.app}, see Appendix, and will apply estimates \eqref{0.28} and \eqref{0.29} only for the approximative solutions $u^N(t)$ which correspond to the initial data $u_0^N$ and external forces $g^N$. Indeed, for every $N$, $u^N$ is a finite-energy solution ($u^N\in C([0,T],L^2(\R^2))$) for which the continuity \eqref{3.cont}
is obvious and, due to Corollary \ref{CorA.app}, we may fix $R=R(u_0,g)\gg1$ in such way that \eqref{0.30} will hold {\it uniformly} for all $N$.
Thus, we have justified the uniform with respect to $N$ estimate
\begin{equation}\label{3.un}
Z_{R,x_0}(u^N(t))+\int_t^{t+1}\|u^N(t)\|^2_{W^{1,2}_{\theta_{R,x_0}}}\,dt\le C(Z_{R,x_0}(u_0)+ Z_{R,x_0}(g)),
\end{equation}
where $C$ is independent of $t\in\R_+$ and $x_0\in\R^2$. Passing now to the limit $N\to\infty$ (in the local topology of $L^2_{loc}(\R^2)$),
we construct the desired weak solution $u$ of the Navier-Stokes equations which will satisfy  \eqref{3.global} and finish the proof of the theorem.
\end{proof}
Remind that we have constructed a weak solution $u(t)$ of the Navier-Stokes problem \eqref{0.1} which satisfies the key estimate \eqref{3.global}. Up to the moment, we do not know whether or not all weak solutions satisfy it. However, the next theorem shows that the weak solution is unique and, by this reason, \eqref{3.global} holds for all solutions.

\begin{theorem}\label{Th3.unique} The weak solution in the sense of Definition \ref{Def0.1} is unique. Moreover, for any two solutions $u_1$ and $u_2$ of problem \eqref{0.1}, we have
\begin{equation}\label{uniqueness}
\|u_1(t)-u_2(t)\|_{L^2_{\theta_{R,x_0}}}\le C_T\|u_1(0)-u_2(0)\|_{L^2_{\theta_{R,x_0}}},
\end{equation}
where the constant $C$ depends on $T$, $R$, $u_1$ and $u_2$, but is independent of $x_0\in\R^2$.
\end{theorem}
\begin{proof} Let $u_1(t)$ and $u_2(t)$ be two weak solutions of \eqref{0.1} and $v(t):=u_1(t)-u_2(t)$. Then, this function solves
\begin{equation}\label{2.unieq}
\Dt v+(u_1,\Nx) v+(v,\Nx) u_2=\Dx v-\alpha v+\nabla P(w_1-w_2),\ \ \divv v=0,
\end{equation}
where $w_i=u_i\otimes u_i$. Multiplying this equation on $\varphi_{R,x_0}v$ (where the cut-off functions $\varphi_{R,x_0}(x)$ are the same as in \eqref{0.11} and \eqref{0.12}), after the integration by parts and straightforward estimates, we get
\begin{multline}\label{3.estdif}
\frac12\frac{d}{dt}\|v\|^2_{L^2_{\varphi_{R,x_0}}}+\alpha\|v\|^2_{L^2_{\varphi_{R,x_0}}}+\|\Nx (v\varphi^{1/2}_{R,x_0})\|^2_{L^2(B^{2R}_{x_0})}\le
C\|\Nx u_2\|_{L^2(B^{2R}_{x_0})}\|v\varphi^{1/2}_{R,x_0}\|^2_{L^4}+\\+C R^{-1}\|u_1\|_{L^2(B^{2R}_{x_0})}\|v\|^2_{L^4(B^{2R}_{x_0})}+|(\nabla P(w_1-w_2),v\varphi_{R,x_0})|+CR^{-1}\|v\|^2_{L^2(B^{2R}_{x_0})}.
\end{multline}
We start with estimating the most complicated term containing pressure. Using Lemma \ref{Lem0.2} together with the fact that $\|u_i(t)\|^2_{L^2(B^R_{x_0})}$ is bounded (by the definition of a weak solution), analogously to \eqref{0.huuge}, we have
\begin{multline}
|(\nabla P(w_1-w_2),v\varphi_{R,x_0})|\le C\int_{\R^2}\theta_{R,x_0}(x)\|w_1-w_2\|_{L^{4/3}(B^R_x)}\,dx\cdot\|v\|^2_{L^4(B^{2R}_{x_0})}\le\\\le
C\int_{\R^2}\theta_{R,x_0}(x)(\|u_1\|_{L^2(B^R_x)}+\|u_2\|_{L^2(B^R_x)})\|v\|_{L^4(B^R_x)}\,dx\cdot\|v\|_{L^4(B^{2R}_{x_0})}\le\\\le
 C_{u_1,u_2}\int_{\R^2}\theta_{R,x_0}(x)\|v\|_{L^4(B^R_x)}\,dx\cdot\|v\|_{L^4(B^{2R}_{x_0})}\le\\\le C\|v\|^2_{L^4(B^{2R}_{x_0})}+
 C_{u_1,u_2}^2R^{-1}\int_{\R^2}\theta_{R,x_0}(x)\|v\|^2_{L^4(B^R_{x})}\,dx\le\\\le C\|v\|^2_{L^4(B^{2R}_{x_0})}+
 C R\int_{\R^2}\theta_{R,x_0}(x)\|v\|^2_{L^4(B^R_{x})}\,dx,
\end{multline}
where the constant $C$ depends on the $L^2_b$-norms of the solutions $u_1$ and $u_2$, but is independent of $R$. We now estimate the more standard first term in the right-hand side of \eqref{3.estdif} using the interpolation inequality
\begin{equation}\label{3.int-lad}
\|U\|_{L^4}^2\le C\|U\|_{L^2}\|\Nx U\|_{L^2}.
\end{equation}
This gives
\begin{multline}
C\|\Nx u_2\|_{L^2(B^{2R}_{x_0})}\|v\varphi^{1/2}_{R,x_0}\|^2_{L^4}\le C_1\|\Nx u_2\|_{L^2(B^{2R}_{x_0})}\|v\|_{L^2_{\varphi_{R,x_0}}}
\|\Nx(v\varphi^{1/2}_{R,x_0})\|_{L^2}\le\\\le C_2\|\Nx u_2\|^2_{L^2(B^{2R}_{x_0})}\|v\|^2_{L^2_{\varphi_{R,x_0}}}+\frac12\|v\varphi^{1/2}_{R,x_0}\|^2_{L^2}.
\end{multline}
Inserting the obtained estimates into the right-hand side of \eqref{3.estdif} and using that the $L^2_b$-norms of $u_i(t)$ are under the control, we end up with
\begin{multline}\label{3.estgr}
\frac d{dt}\|v\|^2_{L^2_{\varphi_{R,x_0}}}-C\|\Nx u_2\|^2_{L^2(B^{2R}_{x_0})}\|v\|^2_{L^2_{\varphi_{R,x_0}}}+\|\Nx v\|^2_{L^2(B^R_{x_0})}\le\\\le C_R\|v\|^2_{L^4(B^{2R}_{x_0})}+
C_R\int_{\R^2}\theta_{R,x_0}(x)\|v\|^2_{L^4(B^{R}_x)}\,dx,
\end{multline}
where the constants $C_R$ depend on $R$ and on the $L^2_b$-norms of the solutions $u_1$ and $u_2$, but are independent of $x_0\in\R^2$. Moreover, by the definition of a weak solutions,
$$
\int_0^T\|\Nx u_2(t)\|^2_{L^2(B^{2R}_{x_0})}\,dt\le C_T,
$$
where $C_T$ depends on $T$, but is independent of $x_0$. Thus, we can apply the Gronwall's inequality to \eqref{3.estgr} which gives
\begin{multline}\label{3.after}
\|v(t)\|^2_{L^2_{\varphi_{R,x_0}}}+\beta\int_0^t\|\Nx v(s)\|^2_{L^2(B^{2R}_{x_0})}\,ds\le C_{R,T}\|v(0)\|^2_{L^2_{\varphi_{R,x_0}}}+\\+
C_{R,T}\int_0^t\|v(s)\|^2_{L^4(B^{2R}_{x_0})}\,ds+C_{R,T}\int_0^t\int_{\R^2}\theta_{R,x_0}(x)\|v(s)\|^2_{L^4(B^R_x)}\,ds\,dt,
\end{multline}
where the constants $\beta$ and $C_{R,T}$ are independent of $x_0\in\R^2$. This allows us to multiply \eqref{3.after} by $\theta_{R,y_0}(x_0)$, integrate over $x_0\in\R^2$ and use \eqref{1.thetakey} and \eqref{1.ul-wR} to obtain
\begin{multline}
\|v(t)\|^2_{L^2_{\theta_{R,x_0}}}+\beta\int_0^t\|\Nx v(s)\|^2_{L^2_{\theta_{R,x_0}}}\,ds\le\\\le C_{R,T}\|v(0)\|^2_{L^2_{\theta_{R,x_0}}}+ C_{R,T}\int_0^t\int_{\R^2}\theta_{R,x_0}(x)\|v\|^2_{L^4(B^R_x)}\,dx\,ds.
\end{multline}
Estimating the $L^4$-norm in the right-hand side with the help of \eqref{3.int-lad}, we finally arrive at
\begin{equation}
\|v(t)\|^2_{L^2_{\theta_{R,x_0}}}+\frac\beta2\int_0^t\|\Nx v(s)\|^2_{L^2_{\theta_{R,x_0}}}\,ds\le C_{R,T}\|v(0)\|^2_{L^2_{\theta_{R,x_0}}}+ C_{R,T}\int_0^t\|v(s)\|^2_{L^2_{\theta_{R,x_0}}}\,ds.
\end{equation}
Applying once more the Gronwall's inequality, we derive the desired estimate \eqref{uniqueness} and finish the proof of the theorem.
\end{proof}
\begin{remark} Note that the uniqueness is proved {\it without} assuming the spatial decay condition \eqref{3.gu} and, therefore, holds for any weak solutions of the Navier-Stokes problem in the sense of Definition \ref{Def0.1}.
\end{remark}

\section{Global well-posedness: the case of  spatially non-decaying initial data}\label{s4}

The aim of this section is to obtain the analogue of Theorem \ref{Th0.1} {\it without} the extra spatial decay condition \eqref{3.gu} on the initial data. Recall, that estimate \eqref{3.global} has been obtained using the purely energetic methods (weighted $L^2$-estimates) without any use of the vorticity estimates (which are traditionally the key technical tools for studying the Navier-Stokes equations in the whole space $\R^2$). However, this energy method requires some decay of the initial data at infinity (although the rate of this decay may be arbitrarily slow) and
 we do not know how to remove this assumption remaining in the class of weighted energy estimates.
 \par
 In the present section, we will show that the extra decay assumption \eqref{3.gu} can be nevertheless removed if we combine the method presented in the proof of Theorem \ref{Th0.1} with the classical vorticity estimates. Since these estimates require the initial data $u_0$ to be more regular than just $u_0\in L^2_b(\R^2)$, we start with reminding the results on the {\it local} solvability and local smoothing property for the solutions for the Navier-Stokes equation \eqref{0.1}.

 \begin{proposition}\label{Prop4.local} Let $u_0,g\in L^2_b(\R^2)$ be divergence free. Then, there exists time moment $T=T(\|u_0\|_{L^2_b},\|g\|_{L^2_b})>0$ and a unique weak solution of problem \eqref{0.1} defined on the time interval $t\in[0,T]$.
 \end{proposition}
 \begin{proof} Indeed, the uniqueness of the weak solution is verified in Theorem \ref{Th3.unique}, so we only need the local existence. To this end, as usual, it is sufficient to verify the proper {\it local} a priori estimate for the weak solution on a small time interval $t\in[0,T]$. To obtain such an estimate, in turn, it is sufficient to integrate estimate \eqref{0.27} (which holds for every weak solution) in time and take a supremum over $x_0\in\R^2$. Then, after the straightforward estimates, we end up with
 \begin{equation}\label{4.local}
 \|u(t)\|_{L^2_{b,R}}^2+\sup_{x_0\in\R^2}\int_0^t\|\Nx u(s)\|^2_{L^2_{\theta_{R,x_0}}}\,ds\le C(\|u_0\|^2_{L^2_{b,R}}+\|g\|^2_{L^2_{b,R}})+C\int_0^t\|u(s)\|^4_{L^2_{b,R}}\,ds,
 \end{equation}
 for some constant $C$ which depends on $R$. Estimate \eqref{4.local} is enough to conclude that there exists $T=T(\|u_0\|_{L^2_{b,R}},\|g\|_{L^2_{b,R}})$ such that
 \begin{equation}\label{4.locest}
 \|u\|_{L^\infty([0,T],L^2_{b,R})}+\|u\|_{L^2_b([0,T],W^{1,2}_b)}\le C(\|u_0\|_{L^2_{b,R}}+\|g\|_{L^2_{b,R}})
 \end{equation}
 and the proposition is proved.
 \end{proof}

\begin{proposition}\label{Prop4.smoo} Let the assumptions of Proposition \ref{Prop4.local} hold. Then, the local weak solution $u(t)$ becomes smoother: $u(t)\in W^{1,2}_b(\R^2)$ for all $t>0$ and the following estimate holds:
\begin{equation}\label{4.locsm}
\|u\|_{L^\infty([t,T],W^{1,2}_b)}+\|u\|_{L^2_b([t,T],W^{2,2}_b)}\le Ct^{-1/2}Q(\|g\|_{L^2_b}+\|u_0\|_{L^2_b}),
\end{equation}
where $T\ll1$ is the same as in Proposition \ref{Prop4.local}, $t\in(0,T]$, and the monotone increasing function $Q$ is independent of the concrete choice of $u(t)$.
\end{proposition}
\begin{proof} To derive estimate \eqref{4.locsm}, we multiply equation \eqref{0.1} by $t\sum_{i}\partial_{x_i}(\varphi_{R,x_0}\partial_{x_i} u)$ and integrate over $x$. Then, after the standard transformations, we get
\begin{multline}\label{4.h1}
\frac d{dt}(t\|\Nx u\|^2_{L^2_{\varphi_{R,x_0}}})+2t\|\Dx u\|^2_{L^2_{\varphi_{R,x_0}}}\le C(t\|u\|_{L^2(B^{2R}_{x_0})}+t+1)\|u\|^2_{W^{1,2}(B^{2R}_{x_0})}+\\+Ct|(\Nx P(u\otimes u,),\sum_{i}\partial_{x_i}(\varphi_{R,x_0}\partial_{x_i} u))|,
\end{multline}
where $C$ may depend on $R$, but is independent of $x_0$. Here we have implicitly used that in 2D $(\Dx u,(u,\Nx)u)\equiv0$ for every divergence free function and, therefore, the leading part of the inertial term in \eqref{4.h1} disappears and the lower order remainder (which appears due to the presence of the weight) can be easily estimated by the first term in the right-hand side of \eqref{4.h1}. Thus, we only need to estimate the term containing pressure in the right-hand side of \eqref{4.h1}. To this end, we observe that $\Dx u$ is divergent free, so due to  Lemma \ref{Lem0.2},
\begin{multline}\label{4.press}
Ct|(\nabla P(w),\varphi_{R,x_0}\Dx u)|\le C_1 t\|\Dx u\|_{L^2_{\varphi_{R,x_0}}}\int_{\R^2}\theta_{R,x_0}(x)\|u\|^2_{L^4(B^R_{x})}\,dx\le \\\le
t\|\Dx u\|^2_{L^2_{\varphi_{R,x_0}}}+C_2\(\int_{\R^2}\theta_{R,x_0}(x)\|u\|_{L^2(B^{R}_x)}\|u\|_{W^{1,2}(B^R_x)}\,dx\)^2\le\\\le
t\|\Dx u\|^2_{L^2_{\varphi_{R,x_0}}}+C_3 Q(\|u_0\|_{L^2_b}+\|g\|_{L^2_b})\int_{\R^2}\theta_{R,x_0}(x)\|u\|^2_{W^{1,2}(B^R_{x})}\,dx,
\end{multline}
where the constants $C_i$ may depend on $R$, but are independent of $x_0\in\R^2$ (here we have also used \eqref{4.locest} in order to control the $L^2_b$-norm of $u$ and the fact that $t\le T$ and $T$ is small). Analogously, using that $\partial_{x_i} u$, $i=1,2$ are both divergence free and using Lemma \ref{Lem0.2} with $\varphi_{R,x_0}$ replaced by $\partial_{x_i}\varphi_{R,x_0}$, we have
\begin{multline}\label{4.press1}
Ct|(\nabla P(w),\Nx\varphi_{R,x_0}\cdot\Nx u)|\le C_1t\|\Nx u\|_{L^2(B^{2R}_{x_0})}\int_{\R^2}\theta_{R,x_0}(x)\|u\|^2_{L^4(B^R_{x})}\,dx\le \\\le
t\|\Nx u\|^2_{L^2(B^{2R}_{x_0})}+C_2\(\int_{\R^2}\theta_{R,x_0}(x)\|u\|_{L^2(B^{R}_x)}\|u\|_{W^{1,2}(B^R_x)}\,dx\)^2\le\\\le
C\|\Nx u\|^2_{L^2(B^{2R}_{x_0})}+C_3 Q(\|u_0\|_{L^2_b}+\|g\|_{L^2_b})\int_{\R^2}\theta_{R,x_0}(x)\|u\|^2_{W^{1,2}(B^R_{x})}\,dx.
\end{multline}
Inserting \eqref{4.press} and \eqref{4.press1} into the right-hand side of equation \eqref{4.h1} and integrating in time over $[0,t]$, we get
\begin{multline}\label{4.h1int}
t\|\Nx u(t)\|_{L^2_{\varphi_{R,x_0}}}^2+\int_0^t t\|\Dx u(s)\|^2_{L^2_{\varphi_{R,x_0}}}\,ds\le\\\le CQ(\|u_0\|_{L^2_b}+\|g\|_{L^2_b})\(\int_0^t\|u(s)\|^2_{W^{1,2}(B^{2R}_{x_0})}+\int_{\R^2}\theta_{R,x_0}(x)\int_0^t\|u(s)\|^2_{W^{1,2}(B^R_x)}\,ds\,dx\)\le\\\le
Q_1(\|u_0\|_{L^2_b}+\|g\|_{L^2_b}),
\end{multline}
where we have used \eqref{4.locest} again in order to estimate the time integral of the $W^{1,2}$-norm of $u$. Since the last estimate is uniform with respect to $x_0\in\R^2$, taking the supremum with respect to $x_0$, we end up with \eqref{4.locsm} and finish the proof of the proposition.
\end{proof}
We now make an extra assumption on the smoothness of the external forces $g$, namely, we assume that
\begin{equation}\label{4.glinf}
\rot g:=\partial_{x_1}g_2-\partial_{x_2}g_1\in L^\infty(\R^2)
\end{equation}
and remind that the vorticity $\omega:=\rot u$ satisfies the heat equation with the transport term:
\begin{equation}\label{4.vor}
\Dt \omega+(u,\Nx)\omega+\alpha\omega-\Dx\omega=\rot g,\ \ \omega\big|_{t=0}=\rot u_0.
\end{equation}
The crucial property is that the vorticity equation possesses the maximum/comparison principle which allows us to control the $L^\infty$-norm of the vorticity $\omega$.
\begin{proposition}\label{Prop4.max} Let the assumptions of Proposition \ref{Prop4.local} holds and let, in addition $\rot u_0\in L^\infty(\R^2)$ and \eqref{4.glinf} is satisfied. Then, the following estimate holds for $\omega:=\rot u$:
\begin{equation}\label{4.max}
\|\omega(t)\|_{L^\infty}\le \|\rot u_0\|_{L^\infty}\,e^{-\alpha t}+\frac1\alpha\|\rot g\|_{L^\infty}.
\end{equation}
\end{proposition}
Indeed, estimate \eqref{4.max} is an immediate corollary of the comparison principle, see e.g.,\cite{ST07} for more details.
\par
Finally, if the initial data $u_0\in L^2_b(\R^2)$ only as the next proposition shows, $\omega(t)\in L^\infty(\R^2)$ for  $t>0$ and Proposition \ref{Prop4.max} can be nevertheless used.

\begin{proposition}\label{Prop4.linfsm} Let the assumptions of Proposition \ref{Prop4.local} holds and let, in addition, \eqref{4.glinf} be satisfied.
Then, for any weak solution $u(t)$ of problem \eqref{0.1}, $\omega(t)\in L^\infty(\R^2)$ for all $t\in(0,T]$ and the following estimate holds:
\begin{equation}
\|\omega(t)\|_{L^\infty(\R^2)}\le t^{-N}Q(\|u_0\|_{L^2_b}+\|g\|_{L^2_b}+\|\rot g\|_{L^\infty}),
\end{equation}
for some positive $N$ and monotone function $Q$ which are independent of $t$, $u_0$ and $g$.
\end{proposition}
Indeed, due to Proposition \ref{Prop4.smoo}, $\omega(t)\in L^2_b(\R^2)$ and $u(t)\in W^{1,2}_b(\R^2)$ for $t>0$. The further regularity of $\omega(t)$ can be now obtained by the classical smoothing estimates for the heat equation \eqref{4.vor} using, say, the Moser iterations.
\par
We are now ready to state and prove the main result of this section.

\begin{theorem}\label{Th4.main} Let $u_0,g\in L^2_b(\R^2)$ be divergent free and let, in addition, $\rot g\in L^\infty(\R^2)$. Then the unique weak solution $u(t)$ of the Navier-Stokes problem \eqref{0.1} exists globally in time and the following estimate holds:
\begin{equation}\label{4.globalest}
\|u(t)\|_{L^2_b}\le Q(\|u_0\|_{L^2_b})+Q(\|g\|_{L^2_b}+\|\rot g\|_{L^\infty}),
\end{equation}
where $Q$ is independent of time and, therefore, every weak solution is globally bounded in time.
\end{theorem}
\begin{proof} We first note that, due to Proposition \ref{Prop4.linfsm}, we may assume without loss of generality that $\rot u_0\in L^\infty(\R^2)$, so we may use the vorticity estimate \eqref{4.max} starting from $t=0$. The idea of the proof is to estimate the $L^3$-norm in the right-hand side of \eqref{3.main} (see the proof of Theorem \ref{Th0.1}) in a better way using the vorticity estimate \eqref{4.max} and the proper interpolation. To this end, we need the following lemma.

\begin{lemma}\label{Lem4.key}
  Let the vector field $u\in [W^{1,2}_0(B^{2R}_{x_0})]^2$ be such that $\divv u,\rot u\in L^\infty(B^{2R}_{x_0})$. Then,
\begin{equation}\label{4.keyest}
\|u\|_{L^3(B^{2R}_{x_0})}\le C\|u\|_{L^2(B^{2R}_{x_0})}^{5/6}\(\|\rot u\|_{L^\infty(B^{2R}_{x_0})}+\|\divv u\|_{L^\infty(B^{2R}_{x_0})}\)^{1/6}, 
\end{equation}
where the constant $C$ is independent of $R$ and $x_0$. Moreover, for any $2<p<\infty$,
\begin{equation}\label{4.keyest1}
\|u\|_{L^\infty(B^{2R}_{x_0})}\le C\|u\|_{L^2(B^{2R}_{x_0})}^{\theta}\(\|\rot u\|_{L^p(B^{2R}_{x_0})}+\|\divv u\|_{L^p(B^{2R}_{x_0})}\)^{1-\theta},\ \ \ \theta=\frac12-\frac1{2(p-1)},
\end{equation}
where $C$ may depend on $p$, but is independent of $R$ and $x_0\in\R^2$. 
\end{lemma}
For the proof of the lemma see Appendix 2.
\par 
We are now returning to the proof of the theorem. Using \eqref{4.keyest} and the cut-off functions \eqref{0.11} and \eqref{0.12}, we
estimate the $L^3$-norm of the solution as follows:
\begin{multline}\label{4.estcorrect}
\|u\|_{L^3(B^R_{x_0})}^3\le \|u\varphi_{R,x_0}\|_{L^2(B^{2R}_{x_0})}^3\le\\\le C\|u\|_{L^2(B^{2R}_{x_0})}^{5/2}\(\|\rot(\varphi_{R,x_0}u)\|_{L^\infty}+\|\divv(\varphi_{R,x_0} u)\|_{L^\infty}\)^{1/2}\le\\\le C\|u\|_{L^2(B^{2R}_{x_0})}^{5/2}
\( CR^{-1}\|u\|_{L^\infty(B^{2R}_{x_0})}+C\|\omega\|_{L^\infty(B^{2R}_{x_0})}\)^{1/2}\le\\\le
C\|u\|_{L^2(B^{4R}_{x_0})}^{5/2}\|\omega\|^{1/2}_{L^\infty}+CR^{-1/2}\|u\|_{L^2(B^{4R}_{x_0})}^{5/2}\|u\|_{L^\infty(B^{2R}_{x_0})}^{1/2},
\end{multline}
where the constant $C$ is independent of $R$ and $x_0$. Analogously, using estimate \eqref{4.keyest1}, say, with $p=4$, we have
\begin{multline*}
\|u\|_{L^\infty(B^{2R}_{x_0})}\le \|\varphi_{2R,x_0}u\|_{L^\infty}\le C\|u\|_{L^2(B^{4R}_{x_0})}^{1/3}\(\|w\|_{L^4(B^{4R}_{x_0})}+R^{-1}\|u\|_{L^4(B^{4R}_{x_0})}\)^{2/3}\le\\\le
C\|u\|_{W^{1,2}(B^{4R}_{x_0})}^{1/3}\(R^{1/2}\|\omega\|_{L^\infty}+R^{-1}\|u\|_{W^{1,2}(B^{4R}_{x_0})}\)^{2/3}
\le\\\le C R^{1/2}\|\omega\|_{L^\infty}+C\|u\|_{W^{1,2}(B^{4R}_{x_0})},
\end{multline*}
where the constant $C$ is independent of $R$ and $x_0$ and we have implicitly used the embedding theorem $W^{1,2}\subset L^4$.
Inserting this estimate into the right-hand side of \eqref{4.estcorrect}, we arrive at
\begin{multline}
\|u\|_{L^3(B^R_{x_0})}^3\le C\|u\|_{L^2(B^{4R}_{x_0})}^{5/2}\|\omega\|^{1/2}_{L^\infty}+CR^{-1/2}\|u\|_{L^2(B^{4R}_{x_0})}^{5/2}\|u\|_{W^{1,2}(B^{4R}_{x_0})}^{1/2}\le\\\le C\|u\|_{W^{1,2}(B^{4R}_{x_0})}^2\(\|\omega\|_{L^\infty}^{1/2}\|u\|^{1/2}_{L^2(B^{4R}_{x_0})}+R^{-1/2}\|u\|_{L^2(B^{4R}_{x_0})}\)\le\\\le C R^{1/2}\(R^{-1}\|u\|_{L^2_{b,R}}+\|\omega\|_{L^\infty}\)\|u\|_{W^{1,2}(B^{4R}_{x_0})}^2,
\end{multline}
where the constant $C$ is independent of $R$ and $x_0$. Therefore, using \eqref{1.eqtheta}, we have
\begin{equation}
\int_{\R^3}\theta_{R,x_0}(x)\|u\|^3_{L^3(B^R_{x})}\,dx\le CR^{1/2}\(R^{-1}\|u\|_{L^2_{b,R}}+\|\omega\|_{L^\infty}\)\int_{\R^2}\theta_{R,x_0}(x)\|u\|^2_{W^{1,2}(B^R_{x})}\,dx.
\end{equation}
Inserting this estimate into the right-hand side of \eqref{3.main}, we get
\begin{multline}\label{4.main}
\frac d{dt}Z_{R,x_0}(u(t))+\beta Z_{R,x_0}(u(t))+\\+\(2\beta-KR^{-1/2}(R^{-1}\|u\|_{L^2_{b,R}}+\|\omega\|_{L^\infty})\) \int_{x\in\R^2}\theta_{R,x_0}(x)\|u\|_{W^{1,2}(B^R_x)}^2\,dx\le C Z_{R,x_0}(g),
\end{multline}
where the positive constants $C$, $K$ and $\beta$ are independent of $R$ and $x_0$.
\par
Analogously to the proof of Theorem \ref{Th0.1}, if the parameter $R$ is chosen in such way that
\begin{equation}\label{4.loop}
KR^{-1/2}(R^{-1}\|u\|_{L^2_{b,R}}+\|\omega\|_{L^\infty})\le 2\beta,
\end{equation}
the Gronwall estimate applied to \eqref{4.main} gives
\begin{equation}\label{4.Zbound}
Z_{R,x_0}(t)\le Z_{R,x_0}(u_0)e^{-\beta t}+CZ_{R,x_0}(g)\le CR(\|u_0\|_{L^2_b}^2+\|g\|_{L^2_b}^2)
\end{equation}
and, therefore, taking into the account \eqref{4.max} and \eqref{1.ul-wR}, we have
\begin{equation}\label{4.bound}
R^{-1}\|u(t)\|_{L^2_{b,R}}+\|w(t)\|_{L^\infty}\le C_1(\|u_0\|_{L^2_b}+\|\rot u_0\|_{L^\infty}+\|g\|_{L^2_b}+\|\rot g\|_{L^\infty}).
\end{equation}
Finally, arguing as in the end of the proof of Theorem \ref{Th0.1}, we show that both \eqref{4.loop} and \eqref{4.Zbound} are satisfied if the parameter $R$ is chosen in a such way that
\begin{equation}\label{R}
R^{-1/2}=\frac\beta{KC_1}(\|u_0\|_{L^2_b}+\|\rot u_0\|_{L^\infty}+\|g\|_{L^2_b}+\|\rot g\|_{L^\infty})^{-1}.
\end{equation}
This, together with \eqref{4.bound} gives the estimate
\begin{equation}\label{4.best}
\|u(t)\|_{L^2_b}\le \|u(t)\|_{L^2_{b,R}}\le C(\|u_0\|_{L^2_b}+\|\rot u_0\|_{L^\infty}+\|g\|_{L^2_b}+\|\rot g\|_{L^\infty})^3
\end{equation}
and the theorem is proved.
\end{proof}

\section{Dissipativity and attractors}\label{s5}
In the previous section, we have shown that the global weak solution of the damped Navier-Stokes system \eqref{0.1} exists and remains bounded as $t\to\infty$. The aim of the present section is to derive the dissipative analogue of \eqref{4.globalest}. To be more precise, the following theorem can be considered as the main result of the section.

\begin{theorem} \label{Th5.dis} Let the assumptions of Theorem \ref{Th4.main} hold. Then, the unique weak solution $u(t)$ of the damped Navier-Stokes equations \eqref{0.1} possesses the following estimate:
\begin{equation}\label{5.dis}
\|u(t)\|_{L^2_b}\le Q(\|u_0\|_{L^2_b})e^{-\beta t}+Q(\|g\|_{L^2_b}+\|\rot g\|_{L^\infty}),
\end{equation}
where the positive constant $\beta$ and the monotone function $Q$ are independent of $t$ and $u_0$.
\end{theorem}
\begin{proof} As in the proof of Theorem \ref{Th4.main}, we may assume that $\rot u_0\in L^\infty(\R^2)$ and use estimate \eqref{4.max}. Note also that \eqref{4.Zbound} gives us the estimate of the form
\begin{equation}\label{5.wrong}
R^{-2}\|u\|_{L^2_{b,R}}^2\le C(\|u_0\|_{L^2_b}^2e^{-\beta t}+\|g\|^2_{L^2_b})
\end{equation}
which looks as dissipative. However, it is not sufficient to derive the desired estimate \eqref{5.dis} since the parameter $R$ depends on the initial data, see \eqref{R}. The idea of the proof of this theorem is to allow (following to \cite{PZ12}, see also \cite{Z07}) the parameter $R$ to depend on $t$ ($R=R(t)$) and to derive the analogue of \eqref{5.wrong} for the time-dependent $R$ satisfying proper dissipative estimate. To this end, we first need to know how the cut-off ($\varphi_{R,x_0}$) and weight ($\theta_{R,x_0}$) functions depend on the parameter~$R$. We start with the cut-off functions. To satisfy \thetag{0.11} and \eqref{0.12} it is sufficient to take
$$
\varphi_{R,x_0}(x):=\varphi\(\frac xR-x_0\),
$$
where $\varphi$ is a single cut-off function which equals one at $B^1_0$ and zero outside of $B^2_0$ and satisfies \eqref{0.12} with $R=1$. Assuming that $R=R(t)$ is a smooth function, the differentiation gives
\begin{equation}\label{5.phiT}
|\frac \partial{\partial t}\varphi_{R(t),x_0}(x)|\le C\cdot\frac {|R'(t)|}{R(t)}\cdot[\varphi_{R(t),x_0}(x)]^{1/2},
\end{equation}
where the constant $C$ is independent of $R$ and $x_0$. Analogously, the straightforward calculations show that
\begin{equation}\label{5.thetaT}
|\frac\partial{\partial t}\theta_{R(t),x_0}(x)|\le C\cdot \frac{|R'(t)|}{R(t)}\cdot\theta_{R(t),x_0}(x),
\end{equation}
where the constant $C$ is also independent of $R$ and $x_0$.
\par
Using these estimates and arguing exactly as in the derivation of \eqref{4.main}, we end up with the following inequality:
\begin{multline}\label{5.maindis}
\frac d{dt}Z_{R(t),x_0}(u(t))+\beta Z_{R(t),x_0}(u(t))+\(\beta-K_1\cdot\frac{|R'(t)|}{R(t)}\)Z_{R(t),x_0}(u)+\\+\(2\beta-KR(t)^{-1/2}(R(t)^{-1}\|u\|_{L^2_{b,R(t)}}+\|\omega\|_{L^\infty})\) \int_{x\in\R^2}\theta_{R(t),x_0}(x)\|u\|_{W^{1,2}(B^{R(t)}_x)}^2\,dx\le\\\le C Z_{R(t),x_0}(g),
\end{multline}
where the positive constants $C$, $K$, $K_1$ and $\beta$ are independent of $R$ and $x_0$.
\par
Indeed, the estimates of terms which do not involve the time differentiation are identical and only the terms containing time derivatives may cause the difference. At the first step, we multiply \eqref{0.1} by $u\varphi_{R(t),x_0}$ and the term with time derivative now reads
$$
(\Dt u,u\varphi_{R(t),x_0})=\frac12\|u\|_{L^2_{\varphi_{R(t),x_0}}}-\frac12(|u|^2,\Dt\varphi_{R(t),x_0})
$$
and using \eqref{5.phiT}, we estimate the extra term via
$$
|(|u|^2,\Dt\varphi_{R(t),x_0})|\le C\frac{|R'(t)|}{R(t)}\|u\|_{L^2(B^{2R(t)}_{x_0})}^2.
$$
One more extra term we obtain when we multiply equation \eqref{2.13} by $\theta_{R(t),x_0}$ and integrate over $x_0$. Namely,
$$
\int_{x_0\in\R^2}\theta_{R(t),y_0}(x_0)\frac d{dt}\|u\|^2_{L^2_{\varphi_{R(t),x_0}}}\,dx_0=\frac d{dt}Z_{R(t),y_0}(u)-
\int_{\R^2}\Dt\theta_{R(t),y_0}(x_0)\|u\|^2_{L^2_{\varphi_{R(t),x_0}}}\,dx_0
$$
and the extra term here can be estimated using \eqref{5.thetaT} via
$$
|\int_{\R^2}\Dt\theta_{R(t),y_0}(x_0)\|u\|^2_{L^2_{\varphi_{R(t),x_0}}}\,dx_0|\le C\frac{|R'(t)|}{R(t)}Z_{R(t),y_0}(u).
$$
The above two estimate leads in a straightforward way to the extra term in the left-hand side of \eqref{5.maindis}. Thus, under the extra assumption
\begin{equation}\label{5.Rslow}
\frac{|R'(t)|}{R(t)}\le \frac\beta{K_1}
\end{equation}
(which we assume from now on to be true), \eqref{5.maindis} will be identical to \eqref{4.main} and we derive from it that
\begin{equation}\label{5.correct}
R(t)^{-1}\|u\|_{L^2_{b,R(t)}}\le C(\|u_0\|_{L^2_b}e^{-\beta t/2}+\|g\|_{L^2_b})
\end{equation}
if $R(t)$ is such that
\begin{equation}\label{5.Rchoice}
2\beta-KR(t)^{-1/2}(R(t)^{-1}\|u(t)\|_{L^2_{b,R(t)}}+\|\omega(t)\|_{L^\infty})\ge0
\end{equation}
holds for all $t$. Note that, if \eqref{5.correct} holds, due to \eqref{4.max}, we have
\begin{equation}
R(t)^{-1}\|u(t)\|_{L^2_{b,R(t)}}+\|\omega(t)\|_{L^\infty}\le C(\|u_0\|_{L^2_b}+\|\rot u_0\|_{L^\infty})e^{-\beta t/2}+\|g\|_{L^2_b}+\|\rot g\|_{L^\infty}).
\end{equation}
Thus, using the continuity arguments (exactly as in the proof of Theorem \ref{Th0.1}), we may conclude that both \eqref{5.correct} and \eqref{5.Rchoice} hold if we take $R(t)$ as follows:
\begin{equation}\label{5.RT}
R(t)=\frac{\beta^2C^2}{K^2}\((\|u_0\|_{L^2_b}+\|\rot u_0\|_{L^\infty})e^{-\gamma t}+\|g\|_{L^2_b}+\|\rot g\|_{L^\infty}\)^2,
\end{equation}
where the parameter $0<\gamma<\beta/2$ should be chosen small enough to satisfy \eqref{5.Rslow}. Indeed,
$$
\frac{|R'(t)|}{R(t)}=2\gamma \frac{e^{-\gamma t}(\|u_0\|_{L^2_b}+\|\rot u_0\|_{L^\infty})}{(\|u_0\|_{L^2_b}+\|\rot u_0\|_{L^\infty})e^{-\gamma t}+\|g\|_{L^2_b}+\|\rot g\|_{L^\infty}}\le 2\gamma
$$
and \eqref{5.Rslow} is satisfied for $\gamma=\beta\min\{\frac12,\frac1{K_1}\}$.
\par
Thus, estimate \eqref{5.correct} is verified for $R(t)$ satisfying \eqref{5.RT} and, therefore,
\begin{multline}\label{5.fingal}
\|u(t)\|_{L^2_b}\le \|u(t)\|_{L^2_{b,R}}\le\\\le C\((\|u_0\|_{L^2_b}+\|\rot u_0\|_{L^\infty})e^{-\gamma t}+\|g\|_{L^2_b}+\|\rot g\|_{L^\infty}\)^2
\(\|u_0\|_{L^2_b}e^{-\gamma t}+\|g\|_{L^2_b}\)
\end{multline}
which gives the desired estimate \eqref{5.dis} and finishes the proof of the theorem.
\end{proof}
Thus, we have verified that the solution semigroup
\begin{equation}\label{5.sem}
S(t)u_0:=u(t),\ \ S(t): H_b\to H_b,\ \ H_b:=\{u_0\in [L^2_b(\R^2)]^2,\ \divv u_0=0\},
\end{equation}
where $u(t)$ is a unique global weak solution of \eqref{0.1},
is well defined in the phase space $H_b$ and, due to estimate \eqref{5.dis}, it is dissipative in this space. Therefore, we may speak about the associated global attractor. We start with the reminding of the definition of the so-called locally compact attractor which is natural for dissipative systems in unbounded domains, see \cite{MZ08} and references therein for the details.

\begin{definition}\label{Def5.attr} A set $\mathcal A\subset H_b$ is a (locally compact) global attractor for the solution semigroup $S(t)$ iff
\par
1) The set $\mathcal A$ is bounded in $H_b$ and is compact in $H_{loc}:=\{u_0\in[L^2_{loc}(\R^2)]^2,\ \divv u_0=0\}$;
\par
2) It is strictly invariant: $S(t)\mathcal A=\mathcal A$ for all $t>0$;
\par
3) It attracts the images of bounded (in $H_b$) sets in the topology of $H_{loc}$. Namely, for any bounded subset $B\subset H_b$ and any neighborhood
$\mathcal O(\mathcal A)$ of the attractor $\mathcal A$ in the $H_{loc}$-topology, there exists $T=T(B,\mathcal O)$ such that
$$
S(t)B\subset\mathcal O(\mathcal A)
$$
for all $t\ge T$.
\end{definition}
The following corollary gives the existence of such an attractor.
\begin{corollary}\label{Cor5.attr} Let the assumptions of Theorem \ref{Th5.dis} hold. Then the associated solution semigroup possesses a global attractor $\mathcal A$ (in the sense of Definition \ref{Def5.attr}) which is generated by all bounded solutions of \eqref{0.1} defined for all $t\in\R$:
\begin{equation}\label{5.atrattr}
\mathcal A=\mathcal K\big|_{t=0},
\end{equation}
where $\mathcal K\subset L^\infty(\R,H_b)$ is a set of all solutions of \eqref{0.1} defined for all $t\in\R$ and bounded.
\end{corollary}
\begin{proof} According to the abstract attractor's existence theorem, see e.g. \cite{BV89}, we need to verify two properties: 1) the existence of an absorbing set $\mathcal B$ which is bounded in $H_b$ and is compact in $H_{loc}$ and 2) that the operators $S(t)$ are continuous in the $H_{loc}$-topology on $\mathcal B$ for any fixed $t$.
\par
Indeed, due to \eqref{5.dis}, the ball $\mathcal B_R$ in the space $H_b$ will be an absorbing set for the semigroup $S(t)$
if $R$ is large enough
although it is not compact in $H_{loc}$. However, combining \eqref{5.dis} with the smoothing property \eqref{4.locsm}, we see that every
solution started from $u_0\in\mathcal B_R$ will be bounded in the space $H^1_b:=H_b\cap W^{1,2}_b(\R^2)$ if $t\ge T$. Thus, the $R$-ball $\mathcal B_R^1$ of $H^1_b$ will be the desired absorbing set for the semigroup $S(t)$ if $R$ is large enough (obviously, this set is compact in the local topology of $H_{loc}$). Thus, the first property holds.
\par
The second property is an immediate corollary of estimate \eqref{uniqueness} and the elementary fact that the topologies induced on $\mathcal B^1_R$ by the embeddings to $L^2_{loc}(\R^2)$ and $L^2_{\theta_{R,x_0}}(\R^2)$ coincide.
\par
Thus, all assumptions of the abstract attractor's existence theorem are verified and, therefore, $\mathcal A$ exists. The formula \eqref{5.atrattr}  also follows from this theorem and the corollary is proved.
\end{proof}
\begin{remark} It is not difficult to see that the factual smoothness of the attractor is restricted by the smoothness of the external forces $g$ only. In articular, if $g\in C^\infty_b(\R^2)$, then the attractor will be also $C^\infty$-smooth. It also worth to mention that Theorem \ref{Th5.dis} and estimate \eqref{5.fingal} show that
\begin{equation}\label{5.boundsattr}
\|\mathcal A\|_{L^2_b}\le C\|g\|_{L^2_b}\(\|g\|_{L^2_b}+\|\rot g\|_{L^\infty}\)^2,
\end{equation}
where the constant $C$ is independent of the choice of $g$ (but, of course, depends on $\alpha>0$).
\end{remark}
\begin{remark} We note that, in general, the attractor $\mathcal A$ is not compact in the initial topology of $H_b$ (but only in the local topology of $H_{loc}$). However, if the external forces $g$ decay as $|x|\to\infty$ ($g\in \dot L^2_b(\Omega)$), then arguing analogously to \cite{EMZ04}, one can prove that $\mathcal A$ is not only compact in $H_b$, but also has the finite fractal dimension in this space. We return to the more detailed study of this case in the forthcoming paper.
\end{remark}

\section{Classical Navier-Stokes problem: polynomial growth of infinite-energy solutions}\label{s6}

In this concluding section, we apply the technique developed above to the study of spatially non-decaying solutions of the classical Navier-Stokes problem in $\R^2$:
\begin{equation}\label{6.1}
\begin{cases}
\Dt u+(u,\Nx)u=\Dx u+\Nx p+g,\\ \divv u=0,\ \
u\big|_{t=0}=u_0
\end{cases}
\end{equation}
which corresponds to the choice of $\alpha=0$ in \eqref{0.1}. In contrast to the case $\alpha>0$, we cannot expect that every solution is globally bounded in time since for the simplest spatially homogeneous case $g\equiv const\ne0$, we have linearly growing in time solution $u(t)=tg$. The purpose of this section is to prove that all solutions of this problem starting with $u_0\in H_b$ grow at most polynomially in time. Namely, the following result holds.

\begin{theorem}\label{Th6.main} Let the assumptions of Theorem \ref{Th5.dis} hold. Then, every weak solution of the Navier-Stokes problem \eqref{6.1} satisfies the estimate
\begin{equation}\label{6.poly}
\|u(t)\|_{L^2_b}\le Q(\|u_0\|_{L^2_b}+\|g\|_{L^2_b}+\|\rot g\|_{L^\infty})(t+1)^5,\ \ t\in\R_+,
\end{equation}
where the monotone function $Q$ is independent of $t$ and $u$.
\end{theorem}
\begin{proof} As in the proof of Theorem \ref{Th4.main}, we may assume without loss of generality that $\rot u_0\in L^\infty(\R^2)$ and we may use the maximum principle for the vorticity equation. However, the absence of the dissipative term $\alpha \omega$ does not allow to obtain the dissipative or even globally bounded in time estimate and, instead of \eqref{4.max}, we have
\begin{equation}\label{6.max}
\|\omega(t)\|_{L^\infty}\le \|\rot u_0\|_{L^\infty}+t\|\rot g\|_{L^\infty},\ \ t\in\R_+,
\end{equation}
where the right-hand side grows linearly in time.
\par
Similar to Theorem \ref{Th5.dis} the parameter $R$ in weighted energy estimates will depend on $t$, but now $R(t)$ will grow in time in order to compensate the absence of the dissipative term $\alpha u$. So, we fix a big $T$ and consider equation \eqref{6.1} on the time interval $t\in[0,T]$. Following the general scheme described above, we multiply it on $u\varphi_{R,x_0}(x)$ where the parameter $R$ will depend not only on the initial data, but also on $T$ and $x_0\in\R^2$. Then, analogously to \eqref{0.13}, we get
\begin{multline}\label{6.13}
\frac12 \frac d{dt}\|u(t)\|_{L^2_{\varphi_{R,x_0}}}^2+\|\Nx u\|^2_{L^2(B^R_{x_0})}\le\\\le C T\|g\|_{L^2_{\varphi_{R,x_0}}}^2+C(R^{-2}+\frac1T)\|u\|_{L^2(B^{2R}_{x_0})}^2+
 CR^{-1}\|u\|^3_{L^3(B^{2R}_{x_0})}+|(\Nx P(u),\varphi_{R,x_0}u)|,
\end{multline}
where the constant $C$ is independent on $T$ and $R$. Here we have estimated the term with the external forces as follows:
$$
|(g,u\varphi_{R,x_0})|\le\|g\|_{L^2_{\varphi_{R,x_0}}}\|u\|_{L^2_{\varphi_{R,x_0}}}\le T\|g\|^2_{L^2_{\varphi_{R,x_0}}}+T^{-1}\|u\|^2_{L^2_{\varphi_{R,x_0}}}
$$
which looks as an optimal one (in the absence of the dissipative term $\alpha u$) if we consider the solution on the time interval $t\in[0,T]$ only.
\par
The estimates for the pressure term are identical to the damped case, considered before, so, analogously to \eqref{4.main}, we end up with
\begin{multline}\label{6.main}
\frac d{dt}Z_{R,x_0}(u(t))+\\+\(2\beta-KR^{-1/2}(R^{-1}\|u\|_{L^2_{b,R}}+\|\omega\|_{L^\infty})\) \int_{x\in\R^2}\theta_{R,x_0}(x)\|\Nx u\|_{L^{2}(B^R_x)}^2\,dx\le C T Z_{R,x_0}(g)+\\+
\(CR^{-2}+C\frac1T+KR^{-1/2}(R^{-1}\|u\|_{L^2_{b,R}}+\|\omega\|_{L^\infty})\)Z_{R,x_0}(u),
\end{multline}
where the constants $C$ and $K$ are independent on $R$ and $T$. Assuming that $R$ is large enough, we transform it to
\begin{multline}\label{6.main1}
\frac d{dt}Z_{R,x_0}(u(t))+\\+\(2\beta-KR^{-1/2}(R^{-1}\|u\|_{L^2_{b,R}}+\|\omega\|_{L^\infty}+1)\) \int_{x\in\R^2}\theta_{R,x_0}(x)\|\Nx u\|_{L^{2}(B^R_x)}^2\,dx\le C T Z_{R,x_0}(g)+\\+
\(C\frac1T+KR^{-1/2}(R^{-1}\|u\|_{L^2_{b,R}}+\|\omega\|_{L^\infty}+1)\)Z_{R,x_0}(u),
\end{multline}
where $C$ and $K$ are independent of $R$ and $T$. In order to be able to control the last term in the right-hand side, we need to assume that
\begin{equation}\label{6.bad}
KR^{-1/2}(R^{-1}\|u\|_{L^2_{b,R}}+\|\omega\|_{L^\infty}+1)\le 2\beta T^{-1},\ \ t\in[0,T].
\end{equation}
Then, \eqref{6.main1} reads
\begin{equation}\label{6.gr}
\frac d{dt}Z_{R,x_0}(u(t))-\frac {C+2\beta}TZ_{R,x_0}(u(t))\le CTZ_{R,x_0}(g),\ \ t\in[0,T]
\end{equation}
and the Gronwall's inequality together with \eqref{ulR} gives
\begin{equation}\label{6.gr1}
Z_{R,x_0}(u(t)\le C_1(T+1)^2R(\|u_0\|^2_{L^2_b}+\|g\|^2_{L^2_b}),\ \ t\in[0,T],
\end{equation}
where the constant $C_1$ is independent of $R$ and $T$. Thus, using \eqref{6.max} and \eqref{1.ul-wR}, we end up with
\begin{equation}\label{6.good}
R^{-1}\|u(t)\|_{L^2_{b,R}}+\|\omega(t)\|_{L^\infty}+1\le C_2(T+1)(\|u_0\|_{L^2_b}+\|\rot u_0\|_{L^\infty}+\|g\|_{L^2_b}+\|\rot g\|_{L^\infty}+1)
\end{equation}
for all $t\in[0,T]$. Estimate \eqref{6.good} implies \eqref{6.bad} if we fix
\begin{equation}\label{6.R}
R=\(\frac{C_2 K}\beta\)^2(\|u_0\|_{L^2_b}+\|\rot u_0\|_{L^\infty}+\|g\|_{L^2_b}+\|\rot g\|_{L^\infty}+1)^2(T+1)^4.
\end{equation}
Then the approximation and continuity arguments (see the end of the proof of Theorem \ref{Th0.1}) show that both \eqref{6.bad} and \eqref{6.good} are
satisfied under this choice of the parameter $R$. It only remains to note that \eqref{6.good} and \eqref{6.R} imply \eqref{6.poly} and finish the proof of the theorem.
\end{proof}
\begin{remark}\label{Rem6.strange} Although estimate \eqref{6.poly} essentially improves the super-exponential  upper bounds for the growth of $u(t)$ in time known before (see \cite{GMS01} and \cite{ST07}), it is probably still not optimal. Indeed, to the best of our knowledge there are no examples where $u(t)$ grow faster than linear. However, the proved theorem gives a bit more information, namely, it factually shows that the {\it mean value} of $u(t)$ over the large ball of radius $R=(t+1)^4$ grows not faster than linear:
\begin{equation}\label{6.strange}
(t+1)^{-4}\|u\|_{L^2_{b,(t+1)^4}}\le C(t+1),
\end{equation}
where $C$ is independent of time, and this estimate is already {\it optimal} since this quantity grows exactly linearly in time for the spatially homogeneous solution of \eqref{6.1} mentioned at the beginning of the section. Roughly speaking, this shows that a solution with the super-linear growth in time (if it exists) should be "essentially non-homogeneous" in space.
\end{remark}

\section{Appendix 1: Approximations of divergence free vector fields}\label{sA}
The aim of this Appendix is to construct the  sequence of divergence free vector fields with finite supports which approximates a given divergence free vector field   $u\in L^2_b(\R^2)$ or $u\in\dot L^2_b(\R^2)$. To this end, we will essentially use the stream
function $\Theta\in W^{1,2}_{loc}(\R^2)$ which generates the  vector field $u$ via $u=\nabla^\perp\Theta$. However, for general divergent free $u\in L^2_b(\R^2)$, the associated stream function does not belong to $W^{1,2}_b(\R^2)$ and may grow when $|x|\to\infty$. The next standard lemma shows that this growth is at most linear.

\begin{lemma}\label{LemA.stream} Let $u\in L^2_{loc}(\R^2)$ be a divergent free vector field. Then, the associated stream function $\Theta\in W^{1,2}_{loc}(\R^2)$ can be chosen in such way that
\begin{equation}\label{A.est}
\|\Theta\|_{L^2(B^1_{(r_1,r_2)})}\le C\int_0^{r_1}\|u\|_{L^2(B^2_{(s,0)})}\,ds+C\int_0^{r_2}\|u\|_{L^2(B^2_{(r_1,s)})}\,ds,
\end{equation}
where the constant $C$ is independent on $u$ and $(r_1,r_2)\in\R^2$.
\end{lemma}
\begin{proof} Indeed, for smooth $u$, the stream function $\Theta$ can be restored via the curvilinear integral
\begin{equation}\label{A.path}
\Theta(x,y)=\int_{\gamma(x,y)}-u_2(x,y)\,dx+u_1(x,y)\,dy,
\end{equation}
where $\gamma(x,y)=\gamma(x,y,x_0,y_0)$ is a curve connecting an arbitrary point $(x_0,y_0)$ with $(x,y)$, for instance, one may take the piece-wise linear path
connecting first $(x_0,y_0)$ with $(x,y_0)$ and then $(x,y_0)$ with $(x,y)$. Here $(x_0,y_0)$ is an arbitrary point (e.g., $(x_0,y_0)=(0,0)$).
In the general case  when $u\in L^2_{loc}(\R^2)$, the path integral \eqref{A.path} may be ill-posed for some exceptional values of $(x_0,y_0)$,
so we average it over $(x_0,y_0)\in B^1_0$ and write
\begin{equation}\label{A.path1}
\Theta(x,y)=\frac1{|B^1_0|}\int_{(x_0,y_0)\in B^1_0}\(\int_{\gamma(x,y,x_0,y_0)}-u_2(x,y)\,dx+u_1(x,y)\,dy\)\,dx_0\,dy_0.
\end{equation}
It is not difficult to check that this integral is well-defined for any $u\in L^2_{loc}(\R^2)$ and indeed $u=\nabla^\perp\Theta$ for any divergent free $u$. Moreover \eqref{A.est} also follows by the straightforward estimates and the lemma is proved.
\end{proof}
\begin{corollary} Let $u\in L^2_b(\R^2)$ be divergent free. Then there exists a stream function $\Theta\in W^{1,2}_{loc}(\R^2)$ such that
\begin{equation}\label{A.bound}
(|x_0|+1)^{-1}\|\Theta\|_{L^2(B^1_{x_0})}\le C\|u\|_{L^2_b},
\end{equation}
where $C$ is independent of $x_0$. If, in addition, $u\in \dot L^2_b(\R^2)$, then
\begin{equation}\label{A.dot}
\lim_{|x_0|\to\infty}(|x_0|+1)^{-1}\|\Theta\|_{L^2(B^1_{x_0})}=0.
\end{equation}
\end{corollary}
Indeed, both statements are immediate corollaries of estimate \eqref{A.est}.
\par
The next corollary can be considered as a main result of the Appendix.

\begin{corollary}\label{CorA.app} Let $u\in L^2_b(\R^2)$ be divergence free. Then, there exists a sequence $u^N\in L^2_b(\R^2)$, $N\in\Bbb N$, of divergence free vector fields such that
\begin{equation}\label{app}
\|u^N\|_{L^2_b}\le C\|u\|_{L^2_b},\ \ \supp u^N\subset B^{2N}_0,\ \ \ u^N(x)=u(x),\ x\in B^N_0,
\end{equation}
where the constant $C$ is independent of $N$ and $u$. If, in addition, $u\in\dot L^2_b(\R^2)$ then
\begin{equation}\label{A.unlim}
\lim_{|x_0|\to\infty}\|u^N\|_{L^2(B^1_{x_0})}=0
\end{equation}
uniformly with respect to $N$.
\end{corollary}
\begin{proof} We may fix
$$
u^N(x):=\nabla^\perp(\Theta(x)\varphi_{N,0}(x))=u(x)\varphi_{R,0}(x)+\Theta(x)\nabla^\perp\varphi_{R,0}(x),
$$
where $\Theta(x)$ is the stream function constructed in Lemma \ref{LemA.stream} and $\varphi_{R,x_0}(x)$ are defined via \eqref{0.11} and \eqref{0.12}. Indeed, since $\varphi_{N,0}(x)=1$ for $|x|\le N$, we have $u^N(x)=u(x)$ for such $x$ and $\varphi_{N,0}(x)=0$ for $|x|\ge2N$ implies that $u^N(x)=0$ for such $x$. Moreover, due to \eqref{0.12}, $|\nabla^\perp\varphi_{R,0}(x)|\le CN^{-1}$ and this gradient is non-zero only if $N\le|x|\le 2N$. Thus, the uniform bounds \eqref{app} in the $L^2_b(\R^2)$ follow from \eqref{A.bound} and the uniform limit \eqref{A.unlim} for $u\in\dot L^2_b(\R^2)$
is an immediate corollary of \eqref{A.dot}. Corollary \ref{CorA.app} is proved.
\end{proof}

\section{Appendix 2: The interpolation inequality}\label{sA2}
The aim of this Appendix is to verify the sharp interpolation inequality stated in Lemma \ref{Lem4.key}. Although this inequality looks more or less standard, it is not easy to find  the precise reference in the literature, so for the convenience of the reader, we sketch its proof here. To this end, we need to remind briefly the definitions and some facts from the theory of Besov spaces and Paley-Littlewood decomposition, see \cite{L02,Tri78} for more detailed exposition.

\begin{definition}\label{DefA2.dyadic} Let $\phi\in C_0^\infty(\R^d)$ be a non-negative cut-off function such that $\phi(\xi)=1$ for $|\xi|\le1/2$ and $\phi(\xi)=0$ for $|\xi|\ge1$ and let $\psi(\xi)=\phi(\xi/2)-\phi(\xi)$. Then, for every $j\in\Bbb Z$, we define operators $S_j,\Delta_j:{\mathcal S}'(\R^d)\to {\mathcal S}'(\R^d)$ as follows:
$$
\widehat{S_j f}:=\phi(\xi/2^j)\widehat f,\ \ \ \widehat {\Delta_jf}:=\psi(\xi/2^j)\widehat f,\ \ f\in {\mathcal S}'(\R^d),
$$
where $\widehat f$ is the Fourier transform of the tempered distribution $f\in {\mathcal S}'(\R^d)$. Then, for any $N\in\Bbb Z$ and any $f\in {\mathcal S}'(\R^d)$,
\begin{equation}\label{A2.PL}
f=S_N f+\sum_{j\ge N}\Delta_j f.
\end{equation}
The distribution $\Delta_j f$ is called $jth$ dyadic block of the distribution $f$ in the Paley-Littlewood decomposition \eqref{A2.PL}. If, in addition,
\begin{equation}\label{A2.0}
\lim_{N\to-\infty}S_Nf=0 \ \ {\rm in} \ \ \mathcal S'(\R^d)
\end{equation}
then the {\it homogeneous} Paley-Littlewood decomposition holds:
\begin{equation}\label{A2.HPL}
f=\sum_{j\in\Bbb Z}\Delta_j f.
\end{equation}
The space of distributions satisfying \eqref{A2.0} is called the space of distributions vanishing at infinity and is denoted by $\mathcal S_0'(\R^d)$.
\end{definition}
Obviously, $L^2(\R^d)\subset \mathcal S'_0(\R^d)$ and by this reason the homogeneous decomposition \eqref{A2.HPL} holds for all functions used in this Appendix. We also remind that the operators $\Delta_j$ commute with differentiation and satisfy the Bernstein type inequalities:
\begin{equation}\label{A2.Ber}
\begin{cases}
1)\  \ \|\Delta_j f\|_{L^p}\le C 2^{-j}\|\nabla\Delta_j f\|_{L^p},\\  2)\ \ \|S_jf\|_{L^q}\le C2^{j(d/p-d/q)} \|S_j f\|_{L^p},
\end{cases}
\end{equation}
where $1\le p\le q\le\infty$ and $C$ is independent of $j$, see e.g., \cite{L02}.
\begin{definition}\label{DefA2.Besov} For any $1\le p,q\le\infty$ and any $\sigma\in\R_+$, we define the Besov space $B^\sigma_{p,q}(\R^d)$ as the subspace of $\mathcal S'(\R^d)$ generated by the following norm:
\begin{equation}\label{A2.Bes}
\|f\|_{B^\sigma_{p,q}}:=\|S_1f\|_{L^p}+\(\sum_{j=1}^\infty(2^{\sigma j}\|\Delta_j f\|_{L^p})^q\)^{1/q}<\infty.
\end{equation}
The homogeneous Besov space $\dot B_{p,q}^\sigma(\R^d)$ is defined via
\begin{equation}\label{A2.HBes}
\|f\|_{\dot B^\sigma_{p,q}}:=\(\sum_{j=-\infty}^\infty(2^{\sigma j}\|\Delta_j f\|_{L^p})^q\)^{1/q}<\infty.
\end{equation}
In contrast to \eqref{A2.Bes} is only a semi-norm and the space $\dot B_{p,q}^\sigma$ is defined by modulo of polynomials (see \cite{L02} for the details). However, it is not essential for us since it is a norm on $\mathbb S_0'(\R^2)$ and we will use these spaces in 
the situation when $f\in L^2(\R^2)\subset \mathbb S'_0(\R^2)$ only.
\par
Finally, the space $BMO(\R^d)$ of functions with bounded mean oscillation is defined by the following semi-norm:
\begin{equation}
\|f\|_{BMO}:=\sup_{B\in\mathcal B}\frac1{|B|}\int_{B}|f(x)-m_Bf|\,dx<\infty,\ \ m_Bf:=\frac1{|B|}\int_Bf(x)\,dx,
\end{equation}
where $\mathcal B$ is a collection of all balls $B_{x_0}^R$, $0<R<\infty$, $x_0\in\R^d$ and $|B|$ is the $d$-dimensional volume of the ball $B$. This space is again defined by modulo of constants and it is again not essential for us since we will work with functions from $L^2(\R^d)\cap BMO(\R^d)$ only.
\end{definition}
The next standard proposition is crucial for what follows.
\begin{proposition}\label{PropA2.ZC} Let $R_j$, $j=1,\cdots, d$, be the Riesz operators defined by
\begin{equation}\label{A2.Riesz}
\widehat{R_j f}=\frac{i\xi_j}{|\xi|}\,\widehat f.
\end{equation}
Then the compositions of Riesz operators $R_{ij}:=R_{i}\circ R_j$ are bounded operators from $L^p(\R^d)$ to $L^p(\R^d)$ and from $L^\infty(\R^d)$ to $BMO(\R^d)$. Moreover, 
$$
BMO(\R^d)\subset \dot B^0_{\infty,\infty}(\R^d)
$$
and, consequently, $R_{ij}$ are bounded operators from $L^\infty(\R^d)$ to $\dot B^0_{\infty,\infty}(\R^d)$ as well.
\end{proposition}
For the proof of this result see, e.g., \cite{L02}. We are now ready to prove Lemma \ref{Lem4.key}
\begin{proof}[Proof of Lemma \ref{Lem4.key}] We first note that \eqref{4.keyest} and \eqref{4.keyest1} are scaling invariant, so we only need to check them for, say, $R=1/2$. We start with the most delicate estimate \eqref{4.keyest}.
\par 
Extending the vector field by zero outside of $B^1_0$, denoting $h_1=\divv u$ and $h_2=\rot u$, after the Fourier transform, we get
\begin{equation}\label{A2.Helm}
\widehat {u_1}=-i\(\frac{\xi_1}{\xi_1^2+\xi_2^2}\widehat {h_1}+\frac{\xi_2}{\xi_1^2+\xi_2^2}\widehat {h_2}\),\ \  \widehat {u_2}=-i\(\frac{\xi_2}{\xi_1^2+\xi_2^2}\widehat {h_1}-\frac{\xi_1}{\xi_1^2+\xi_2^2}\widehat {h_2}\)
\end{equation}
and we see that $\Nx u$ can be expressed through $h_1$ and $h_2$ via the compositions $R_{ij}$ of Riesz operators. Thus, due to Proposition \ref{PropA2.ZC}, $\Nx u\in [\dot B^0_{\infty,\infty}(\R^2)]^4$ and
\begin{equation}\label{B.2}
\|\Nx u\|_{\dot B^0_{\infty,\infty}}\le C\(\|h_1\|_{L^\infty(B^1_0)}+\|h_2\|_{L^\infty(B^1_0)}\).
\end{equation}
 
Thus, using the Bernstein inequalities
$$
\|\Delta_j u\|_{L^p}\le C\|\Delta_j\Nx u\|_{L^p} 2^{-j},
$$
we see that $u\in \dot B^1_{\infty,\infty}(\R^2)$ and
\begin{equation}\label{B.3}
\|u\|_{\dot B^1_{\infty,\infty}(\R^2)}\le C\(\|h_1\|_{L^\infty}+\|h_2\|_{L^\infty}\).
\end{equation}
At the next step, we interpolate between $\dot B^0_{2,2}$ and $\dot B^1_{\infty,\infty}$. Namely, by the H\"older inequality
$$
\|\Delta_j u\|_{L^{12/5}}\le\|\Delta_j u\|_{L^2}^{5/6}\|\Delta_j\|_{L^\infty}^{1/6}
$$
and, therefore,
\begin{multline}\label{B.int}
\|u\|_{\dot B^{\frac16}_{\frac{12}5,\frac{12}5}(\R^2)}=\(\sum_{j\in\Bbb Z} (2^{j/6}\|\Delta_j u\|_{L^{12/5}})^{12/5}\)^{5/12}\le\\\le C\(\sum_{j\in\Bbb Z}\|\Delta_j u\|_{L^2}^2\(2^j\|\Delta_j u\|_{L^\infty}\)^{2/5}\)^{5/12}\le C\|u\|_{\dot B_{2,2}^0}^{5/6}\|u\|_{\dot B_{\infty,\infty}^{1}}^{1/6}.
\end{multline}
Note that, due to the Plancherel equality, $L^2(\R^2)\subset \dot B^0_{2,2}(\R^2)$, so, combining \eqref{B.int} and \eqref{B.3}, we have
\begin{equation}\label{B.besov-hom}
\|u\|_{\dot B^{\frac16}_{\frac{12}5,\frac{12}5}(\R^2)}\le C\|u\|_{L^2}^{5/6}\(\|h_1\|_{L^\infty}+\|h_2\|_{L^\infty}\)^{1/6}.
\end{equation}
On the next step, we obtain the analogue of \eqref{B.besov-hom} for the usual (non-homogeneous) Besov spaces. To this end, we utilize the fact that the supports of $h_1$, $h_2$ and $u$ belong to the unit ball $B^1_0$. By this reason, Proposition \ref{PropA2.ZC} gives
\begin{equation}\label{B.ZC}
\|u\|_{L^2(B^1_0)}\le C\|\Nx u\|_{L^2(B^1_0)}\le C_1\(\|h_1\|_{L^2}+\|h_2\|_{L^2}\)\le C_2\(\|h_1\|_{L^\infty}+\|h_2\|_{L^\infty}\) 
\end{equation}
Therefore, due to the Bernstein inequalities,
$$
\|S_1 u\|_{L^{12/5}}\le C\|S_1 u\|_{L^2}\le C_1\|u\|_{L^2}\le C_2\(\|h_1\|_{L^\infty}+\|h_2\|_{L^\infty}\)
$$
and, together with \eqref{B.ZC} and \eqref{B.besov-hom}, we have
\begin{equation}\label{B.besov}
\|u\|_{B^{\frac16}_{\frac{12}5,\frac{12}5}(\R^2)}\le C\|u\|_{L^2}^{5/6}\(\|h_1\|_{L^\infty}+\|h_2\|_{L^\infty}\)^{1/6}.
\end{equation}
Using now that $l_{q_1}\subset l_{q_2}$ for $q_2>q_1$, we have $B^{\frac16}_{\frac{12}5,\frac{12}5}\subset B^{\frac16}_{\frac{12}5,3}$ and
\begin{equation}\label{B.besov1}
\|u\|_{B^{\frac16}_{\frac{12}5,3}(\R^2)}\le C\|u\|_{L^2}^{5/6}\(\|h_1\|_{L^\infty}+\|h_2\|_{L^\infty}\)^{1/6}.
\end{equation}
Finally, using the embedding theorem
\begin{equation}\label{B.Triebel}
B^{s}_{p,q}(\R^2)\subset L^q(\R^2),\ \ s-\frac2p=-\frac2q,
\end{equation}
(see \cite{Tri78}, page 206) with $s=1/6$, $p=5/12$ and $q=3$, we end up with \eqref{4.keyest}.
\par
Thus, it only remains to verify \eqref{4.keyest1} in the particular case $x_0=0$ and $2R=1$. Indeed, due to Proposition \ref{PropA2.ZC} and expressions \eqref{A2.Helm}, we have
\begin{equation*}
\|u\|_{W^{1,p}(B^1_0)}\le C\|\Nx u\|_{L^p(B^1_0)}\le C_1(\|\divv u\|_{L^p}+\|\rot u\|_{L^p})
\end{equation*}
and estimate \eqref{4.keyest1} is now an immediate corollary of the interpolation inequality
$$
\|u\|_{L^\infty}\le C\|u\|_{L^2}^\theta\|u\|_{W^{1,p}}^{1-\theta},
$$
see \cite{Tri78}. Thus Lemma \ref{Lem4.key} is proved.
\end{proof}

\end{document}